\documentclass[11pt, reqno]{amsart}
\usepackage[utf8]{inputenc}
\usepackage[english]{babel}
 
\usepackage[square,sort,comma,numbers]{natbib}

\usepackage[dvipsnames]{xcolor}
\usepackage{float}

\usepackage{graphicx}
\usepackage{tikz}
\usetikzlibrary{arrows}
\usetikzlibrary{matrix,arrows,decorations.pathmorphing}
\usepackage{pinlabel}

\usepackage{amsmath}
\usepackage{amssymb}
\usepackage{amsthm}

\usepackage{stmaryrd}
\usepackage[all]{xy}
\usepackage[mathcal]{eucal}
\usepackage{verbatim}
\usepackage[top=1.2in,bottom=1.2in,left=1.2in,right=1.2in]{geometry} 
\usepackage{wrapfig}
\usepackage{lipsum}
\usepackage[all]{xy}
\theoremstyle{plain}
\newtheorem{thm}{Theorem}[section]

\newtheorem{lem}[thm]{Lemma}
\newtheorem{obs}[thm]{Observation}
\newtheorem{prop}[thm]{Proposition}
\newtheorem{cor}[thm]{Corollary}
\newtheorem{claim}[thm]{Claim}

\theoremstyle{definition}
\newtheorem{defn}[thm]{Definition}

\newtheorem*{exmp*}{Example}

\newtheorem*{thm_bundle}{Theorem \ref{thm:structure}}

\theoremstyle{remark}
\newtheorem*{rem}{Remark}

\newcommand{\nc}{\newcommand}
\nc{\dmo}{\DeclareMathOperator}
\DeclareMathOperator{\PConf}{PConf}\DeclareMathOperator{\Ker}{Ker}
\DeclareMathOperator{\Homeo}{Homeo}
\DeclareMathOperator{\Stab}{Stab}

\DeclareMathOperator{\Diff}{Diff}
\DeclareMathOperator{\Conf}{Conf}
\DeclareMathOperator{\Fix}{Fix}

\DeclareMathOperator{\SL}{SL}
\DeclareMathOperator{\SO}{SO}
\DeclareMathOperator{\DD}{\mathbb{D}}

\DeclareMathOperator{\Int}{Int}
\DeclareMathOperator{\Ext}{Ext}
\DeclareMathOperator{\res}{Res}

\DeclareMathOperator{\di}{dim}
\DeclareMathOperator{\codi}{codim}

\newcommand{\R}{\mathbb{R}}
\newcommand{\Z}{\mathbb{Z}}

\nc{\para}[1]{\medskip\noindent\textbf{#1.}}
\title{\boldmath Structure theorems for actions of homeomorphism groups}
\address{\newline Department of Mathematics   \newline University of Maryland   \newline College Park, MD, 20742,  USA \newline
Department of Mathematics \newline  Cornell University  \newline Ithaca, NY 14853, USA}
\email{chenlei@umd.edu, k.mann@cornell.edu}
\author{Lei Chen, Kathryn Mann}

\begin{document}

\maketitle

\begin{abstract} 
\setlength{\baselineskip}{1.1em}
We give general classification and structure theorems for actions of groups of homeomorphisms and diffeomorphisms on manifolds, reminiscent of classical results for actions of (locally) compact groups.  This gives a negative answer to Ghys' ``extension problem" for diffeomorphisms of manifolds with boundary, as well as a classification of all homomorphisms $\Homeo_0(M) \to \Homeo_0(N)$ when $\dim(M) = \dim(N)$ (and related results for diffeomorphisms), and a complete classification of actions of $\Homeo_0(S^1)$ on surfaces.  This resolves many problems in a program initiated by Ghys, and gives definitive answers to conjectures of Militon and Hurtado and a question of Rubin.
\end{abstract}

\section{Introduction}
Let $M$ be an oriented, closed manifold.  It is a basic problem to understand the actions of $\Homeo_0(M)$ and $\Diff^r_0(M)$ (the identity components of the group of homeomorphisms or $C^r$ diffeomorphisms of $M$, respectively) on manifolds and other spaces.    This is the analog of {\em representation theory} for these large transformation groups; and our work gives structure theorems and rigidity results towards a classification of all possible actions.  

Natural examples of continuous actions of $\Homeo_0(M)$ and $\Diff^r_0(M)$ on other spaces are induced by modifications of $M$: taking products with other manifolds, considering configuration spaces of points on $M$, taking lifts to covers, and also passing to some fiber bundles. For example, $\Diff^1_0(M)$ acts naturally on the tangent bundle of $M$.
Understanding to what extent these examples form an exhaustive list is a long-standing, basic question.  This article gives a complete (and positive) answer to several precise formulations of this question, including those appearing in \cite{Ghys, Hurtado, Militon, Rubin}.  For example, among other results, we prove the following: 

\begin{thm} \label{thm:transitive}
If $M$ is a connected manifold and $\Homeo_0(M)$ acts transitively on a finite-dimensional connected manifold or CW complex $N$, then $N$ is homeomorphic to a cover of the configuration space $\Conf_n(M)$ of $n$ points  in $M$, and the action on $N$ is induced from the natural action of $\Homeo_0(M)$ on $\Conf_n(M)$. 
\end{thm}

\begin{thm} \label{thm:structure_diff} 
Suppose $M$ is a connected, closed, smooth manifold and $N$ is a connected manifold with $\dim(N) < 2\dim(M)$.  If there exists a nontrivial continuous homomorphism $\rho: \Diff^r_0(M)\to \Diff^s_0(N)$, then the action $\rho$ on $N$ is fixed point free and $N$ is a topological fiber bundle over $M$.  
\end{thm}

Before stating our general results, we motivate this work by listing specific instances of the ``basic question'' advertised above, all of which we eventually answer.  

\subsection{History and motivating questions} \label{sec:history}
Understanding actions of homeomorphism groups and diffeomorphism groups on manifolds has a long history.  
Rubin \cite{Rubin}, asked generally for topological spaces $X$ and $Y$ if there were ``any reasonable assumptions" under which ``the embeddability of $\Homeo(X)$ in $\Homeo(Y)$ will imply that $X$ is some kind of continuous image of $Y$."  Implicit in work of Whittaker \cite{Whittaker} and Filipkiewicz \cite{Filipkiewicz} in the 60s and 80s on automorphisms of $\Homeo_0(M)$ and $\Diff^r_0(M)$, and isomorphisms among such groups, 
is the problem to classify the {\em endomorphisms} of these groups.   

More recent and more specific instances of this question include Ghys' work \cite{Ghys}, where he asks whether the existence of an injective homomorphism $\Diff_0^\infty(M) \to \Diff_0^\infty(N)$ implies that $\dim(M) \leq \dim(N)$.  (This was answered positively by Hurtado in \cite{Hurtado}, but those techniques do not apply to the corresponding problem for groups of homeomorphisms, or of diffeomorphisms of class $C^r$, $r < \infty$.)   Ghys also asked in which cases groups of diffeomorphisms of a manifold with boundary admit {\em extensions to the interior}: homomorphisms $\Diff_0^\infty(\partial M) \to \Diff_0^\infty(M)$ giving a group-theoretic section to the natural ``restrict to boundary" map $\Diff_0^\infty(M) \to \Diff_0^\infty(\partial M)$.  The main result of \cite{Ghys} is the non-existence of such when $M$ is a ball.  

Hurtado \cite[$\S$ 6.1]{Hurtado} asked whether all homomorphisms $\Diff_0^\infty(M) \to \Diff_0^\infty(N)$ (for general $M$, $N$) might be ``built from  pieces" coming from natural bundles over configuration spaces of unordered points on $M$.   Both his and Ghys' questions are equally interesting for homeomorphism groups as well.  
Specific to the homeomorphism case, Militon \cite{Militon} classified the actions of $\Homeo_0(S^1)$ on the torus and closed annulus, and asked whether an analogous result would hold when the target is the open annulus, disc, or 2-sphere.  In the same work, he stated the conjecture that, for a compact manifold $M$, every nontrivial group morphism $\Homeo_0(M) \to \Homeo_0(M)$ is given by conjugating by a homeomorphism.

\subsection{Results} 
We answer all of the questions stated in the subsection above, including a precise formulation of Rubin's question (where ``reasonable assumptions'' in our case are that $X$ and $Y$ are manifolds, and either involve bounds on the dimension $Y$ in terms of that of $X$, or that the action on $Y$ is transitive).   We answer Hurtado's and Ghys' questions in both the Homeo and Diffeo case, and show that Militon's conjecture is false in general (in particular, it fails for all manifolds of negative curvature), but it fails for only one reason, and we can describe all manifolds $M$ for which it does hold.  We also complete Militon's classification of actions of $\Homeo_0(S^1)$ on surfaces, and separately classify all actions of $\Homeo_0(S^n)$ on the $(n+1)$-ball, which is surprisingly different in the case $n=1$ and $n>1$.   
We also discuss actions where $M$ is noncompact, in which case the relevant groups to study are $\Homeo_c(M)$ and $\Diff^r_c(M)$, those homeomorphisms that are {\em compactly supported and isotopic to identity through a compactly supported isotopy}.    Note these groups agree with $\Homeo_0(M)$ and $\Diff^r_0(M)$, respectively, when $M$ is compact.  All manifolds are assumed boundaryless, unless stated otherwise.  

The first step in each of these answers is the following general structure theorem for orbits.   

\begin{thm}[Orbit Classification Theorem] \label{thm:orbit_classification}
Let $M$ be a connected topological manifold.  For any action of $\Homeo_c(M)$ on a finite-dimensional CW complex, every orbit is either a point or the continuous injective image of a cover of a configuration space $\Conf_n(M)$ for some $n$.
If $M$ has a $C^r$ structure, then for any weakly continuous action of $\Diff^r_c(M)$ on a finite-dimensional CW complex by homeomorphisms, every orbit is either a point, or a continuous injective image of a cover of the $r$-jet bundle over $\Conf_n(M)$ under a fiberwise quotient by a subgroup of the extended jet group.  
\end{thm}
See Section \ref{sec:small_quotient} for the definition of extended jet groups.  
Using the continuity result in \cite[Theorem 1.2]{Hurtado}, this theorem gives an immediate positive answer to a precise formulation of Hurtado's question \cite[\S 6.1]{Hurtado}, namely, all actions of $\Diff_c^\infty(M)$ on another manifold are always built from pieces (orbits) that are natural bundles over covers of configuration spaces.  

As mentioned above, $\Conf_n(M)$ denotes the configuration space of $n$ distinct, unlabeled points in $M$, but if $M = S^1$ or $M = \R$, we mean the configuration space of unlabeled points {\em together with a cyclic or linear order}, since $\Homeo_c(M)$ does not act transitively on the space of unordered points.  
Since $\Conf_n(M)$, and the $k$-jet bundles (for $k\leq r$) are manifolds on which $\Homeo_c(M)$ and $\Diff_c^r(M)$ act transitively, all these types examples of orbits do indeed occur naturally.   The next challenge is to 
\begin{enumerate}
\item[a)] determine which covers of configuration spaces can appear (this is discussed in Section 2 along with its relationship with ``point pushing" problems) and 
\item[b)] determine how orbits of various types can be glued together, i.e. how they partition a fixed manifold or CW complex $N$ on which $\Diff_c^r(M)$ acts.   
\end{enumerate}
While a general classification is difficult, we solve these problems in sufficiently many cases to answer all the open questions mentioned above.    

The remainder of the work is organized into three main applications of the Orbit Classification Theorem, stated below.   Although we have stated these as ``applications", in many cases the Orbit Classification Theorem is merely the starting point for the result, with the bulk of the proof requiring many additional techniques.  

\para{Automatic continuity} 
Recent work of Hurtado and the second author \cite{Hurtado, Mann} shows that any abstract homomorphism between groups of $C^\infty$ diffeomorphisms, or from $\Homeo_c(M)$ to any separable topological group (of which all homeomorphism and diffeomorphism groups of manifolds are examples), are necessarily continuous when $M$ is compact, and ``weakly continuous" when $M$ is noncompact.  See Section \ref{sec:ac} for the definition of weak continuity and further discussion.  For actions of $C^r$ diffeomorphism groups, or actions of smooth diffeomorphism groups of by diffeomorphisms of lower regularity, such automatic continuity remains open, so we add continuity as an assumption.    

\para{Application I: Structure theorem in restricted dimension} 
As happens in the classical study of compact transformation groups, the lower the dimension of a space on which a group $G$ acts, the easier it is to classify all actions of $G$ on that space.   This is evident, for instance, in the work of Hsiang and Hsiang \cite{HH1} classifying actions of compact groups on manifolds, where a key assumption is that the manifold have dimension bounded by half the dimension of the group.  

The easiest consequence of Theorem \ref{thm:orbit_classification} comes under very strong assumptions on dimension; it gives a counterexample to (but near proof of) \cite[Conj. 1.1]{Militon} as follows:

\begin{thm}\label{thm:lifting}
Let $M$ be a connected manifold and $N$ a manifold with $\di(N)\le \di(M)$. If there is a nontrivial homomorphism $\rho: \Homeo_c(M)\to \Homeo(N)$, then $\di(N) = \di(M)$, there is a countable collection of covers $M_i$ of $M$, and disjoint embeddings $\phi_i:M_i \to N$ such that $\rho(f)$ agrees with $\phi_i f \phi_i^{-1}$ on the image of $M_i$ and is identity outside $\bigcup_i \phi_i(M_i)$.  
\end{thm}
\noindent  Examples of such lifts to covers abound.  For example, let $S_g$ be a surface of genus $g\geq2$, then $\Homeo_0(S_g)$ lifts to act on every cover of $S_g$.   One way to see this is as follows: fix a hyperbolic metric on $S_g$, so the universal cover can be identified with the Poincar\'e disc.  Then each $f \in  \Homeo_0(S_g)$ admits a unique lift to $\tilde{S_g}$ that extends to a continuous homeomorphism of the closed disc, pointwise fixing the boundary.  This gives a continuous action of $\Homeo_0(S_g)$ on $\tilde{S_g}$.  Embedding the compactification of the disc in another $2$-manifold $N$ (or into $S_g$ itself) and extending the action to be trivial outside the image of the embedded disc gives a nontrivial homomorphism $\Homeo_0(S_g) \to \Homeo_0(N)$.    

We also prove a similar result (see Theorem \ref{thm:structure} below) for continuous actions of diffeomorphism groups; in this case if $M$ is compact then $N$ is necessarily a cover of $M$, and the action on $N$ is transitive. This gives a new proof of Hurtado's classification of actions, given his prior results on continuity.  

\para{Weaker restrictions on dimension} With a weaker restriction on the dimension, we have the following version of a ``slice theorem" for actions of homeomorphism groups.   
\begin{thm}[Structure theorem for group actions by homeomorphisms] \label{flatbundle}
Let $M$ and $N$ be connected manifolds such that $\di(N)<2\di(M)$. If there is an action of $\Homeo_c(M)$ on $N$ without global fixed points, then $N$ has the structure of a {\em generalized flat bundle} over $M$.  When $\di(N)-\di(M)<3$, the fiber $F$ is a manifold as well.  
\end{thm}
\noindent A {\em generalized flat bundle} is a foliated space of the form $(\tilde{M} \times F)/ \pi_1(M)$ where $\pi_1(M)$ acts diagonally by deck transformations on $\tilde{M}$ and on $F$ by some representation to $\Homeo(F)$.  See Section \ref{sec:bundle}.   
There are many examples of such actions on generalized flat bundles.  For instance, we have the following:  
\begin{exmp*}
Let $S_g$ be a surface of genus $g \geq 2$.  There are infinitely many non-conjugate actions of $\Homeo_0(S_g)$ on $S_g \times S^1$. 
One such family may be constructed as follows:  lift the action of $\Homeo_0(S_g)$ to an action on the universal cover $\tilde{S}_g$ as described above, and extend this to an action on the product $\tilde{S}_g \times S^1$ that is trivial on the $S^1$ factor.  Now take a homomorphism $\rho: \pi_1(S_g) \to \Homeo_0(S^1)$ with a global fixed point (or any action of this group on the circle with Euler number zero), and quotient $\tilde{S}_g \times S^1$ by the diagonal action of $\pi_1(S_g)$ acting by deck transformations on the first factor and via $\rho$ on the second.   The quotient space is topologically $S_g \times S^1$ (this is ensured by the section provided by the global fixed point, or more generally by an action with 0 Euler number), and the action of $\Homeo_0(S_g)$ naturally descends to this quotient.  Non-conjugate actions on $S^1$ will produce non-conjugate examples.   
\end{exmp*}

With more regularity, we obtain a stronger result and may remove the assumption on fixed points.  

\begin{prop}[No fixed points] \label{prop:no_fixed}
Suppose $M$ is a connected, closed smooth manifold of dimension at least 3, and suppose that $\Diff^r(M)$ $(0 \leq r \leq \infty)$ acts continuously on a connected manifold $N$ by $C^1 $ diffeomorphisms.  If the action has a global fixed point, then it is trivial.
\end{prop} 

\begin{thm}[Structure theorem for group actions by diffeomorphisms] \label{thm:structure}
Suppose $M$ is a connected, closed, smooth manifold and $N$ is a connected manifold with $\dim(N) < 2\dim(M)$.  Let $0 \leq r \leq \infty$ and $1 \leq s \leq \infty$.   If there exists a nontrivial continuous action $\Diff^r_0(M)\to \Diff^s_0(N)$, then the action is fixed point free, and $N$ is a topological fiber bundle over $M$ where the fibers are $C^s$-submanifolds of $N$.  
\end{thm}
In the case where $\dim(M)=\dim(N)$ and $r = s =\infty$, this recovers \cite[Theorem 1.3]{Hurtado}, with an independent proof.  

\para{Application II: Extension problems} 
Let $W$ be a manifold with boundary $M$. If $W$ has a $C^r$ structure, then there is a natural ``restrict to the boundary" map
\[
\res^r(W,M): \Diff^r_0(W)\to \Diff^r_0(M)
\]
which is surjective. The {\em extension problem}, introduced by Ghys in \cite{Ghys}, asks whether $\res^r(W,M)$ has a {\em group theoretic section}, i.e. a homomorphism $\rho$ such that $\res^r(W,M) \circ \rho$ is the identity map.  In general, one expects that the answer may depend both on $r$ and on the topology of $W$.   For simplicity, we will consider only the case where $W$ and $\partial W = M$ are both connected, although these techniques can be adapted with some work to cover the case where $\partial W$ is not connected.  

One case where the extension problem has a positive answer is for homeomorphism groups of balls and spheres. 
Let $\mathbb{D}^{n+1}$ be the $n+1$ dimensional ball and $S^n$ the $n$-sphere. Then 
\[
\mathbb{D}^{n+1}=\{(x,r)|x\in S^n, r\in [0,1]\}/(x,0)\sim (y,0)\]
and there is a standard ``coning off" action $C: \Homeo_0(S^n)\to \Homeo(\mathbb{D}^{n+1})$ defined by $C(f)(x,r)=(f(x),r)$.  (Going forward, we refer to this as {\em coning}.)  We answer \cite[Question 3.18]{KB} in the following theorem.

\begin{thm}\label{coning}
Let $M$ be a connected manifold with $\dim(M)>1$, and suppose $\pi_1(M)$ has no nontrivial action on the interval (e.g. a group generated by torsion). Then 
$\res^0(W,M)$ has a section if and only if $M=S^n$ and $W=\mathbb{D}^{n+1}$. 
\end{thm}

\begin{thm}
For $n>1$, any section of $\res^0(\mathbb{D}^{n+1},S^n)$ is conjugate to the standard coning.  In fact, any nontrivial action of $\Homeo_0(S^n)$ on $\mathbb{D}^{n+1}$ is conjugate to the standard coning.
\end{thm}
\noindent We actually prove a more general result than Theorem \ref{coning}, see Section \ref{sec:extension}.

Ghys \cite{Ghys} posed the extension problem for the genus $g$ handlebody in the smooth category (see also \cite[Question 3.15, Question 3.19]{KB}).  As a consequence of Theorem \ref{thm:structure}, we not only answer Ghys' question, but also answer the section problem for manifolds with boundary, and any regularity of at least $C^1$.  

\begin{cor}[No differentiable extensions]  \label{cor:no_extension}
Let $W$ be a compact, smooth manifold with boundary $M$, and let $r\ge 1$.  Then $\res^r(W,M): \Diff^r_0(W)\to \Diff^r_0(M)$ does not have a continuous section.   If $r = \infty$, any section is automatically continuous by \cite{Hurtado}, and this hypothesis may be removed. 
\end{cor}

\para{Application III: $\Homeo_0(S^1)$ actions on surfaces} 
In contrast with Theorem \ref{coning}, there are infinitely many non-conjugate extension actions of $\Homeo_0(S^1)$ on $\DD^2$.  
For example, in addition to the standard coning, one may take the action on the open annulus (the configuration space of two marked points, or
$\PConf_2(S^1)$), with one end naturally compactified to a circle and the other to a point.  Section \ref{sec:Militon} is devoted to the general classification problem posed by Militon in \cite{Militon}.  We prove the following. 
\begin{thm}
For each closed, proper set $K \subset [0,1]$ containing $0$, and a continuous function $\lambda: [0,1] - K \to \{0,1\}$ there is an action $\rho_{K,\lambda}: \Homeo_0(S^1)\to \Homeo_0(\mathbb{D}^2)$; this collection of actions has the property that any nontrivial homomorphism $\rho: \Homeo_0(S^1)\to \Homeo_0(\mathbb{D}^2)$ is conjugate to $\rho_{K,\lambda}$ for some $K,\lambda$.  
\end{thm}
The construction of these actions is given in Section \ref{sec:Militon}, following Militon.  Militon's work also 
tells us exactly which of the actions $\rho_{K,\lambda}$ are conjugate (see Section \ref{sec:Militon}), so this gives a complete classification of actions of $\Homeo_0(S^1)$ on the disc.   A similar classification may be obtained by the same methods for actions of $\Homeo_0(S^1)$ on other orientable surfaces.  This proves \cite[Conjecture 2.2]{Militon}.

\para{Structure of the paper}
\begin{itemize}
\item In Section \ref{sec:small_quotient}, we establish a ``small quotient subgroup theorem" that is the main ingredient in our Orbit Classification Theorem. 
\item In Section 3,  we briefly discuss automatic continuity, then prove the Orbit Classification Theorem, give Theorem \ref{thm:transitive} as an easy consequence, and classify homomorphisms between homeomorphism groups when $\dim(M)=\dim(N)$.
 \item In Section 4, we discuss admissible covers.
  \item  In Section 5, we discuss how orbits fit together and prove a structure theorem in the $C^0$ category. 
 \item In Section 6, we study the extension problem in the homeomoprhism case. 
 \item In Section \ref{sec:diffeo_structure}, we prove the structure theorem for actions in the $C^r$ category, $r\ge 1$ and study the extension problem in the differentiable case.
 \item In Section 8, we classify actions of $\Homeo_0(S^1)$ on $\mathbb{D}^2$ and other surfaces.
 \end{itemize}

\para{Acknowledgements}
KM was partially supported by NSF grant DMS-1606254 and a Sloan fellowship, and thanks Caltech for their hospitality. LC was partially supported by NSF grant DMS-2005409. We thank Peter May and Shmuel Weinberger for telling us about Bing's ``dog bone space" and homology manifolds, and Tom Church and Benson Farb for comments on the manuscript. We also thank Emanuel Militon for pointing out a mistake in an early version of the proof of Lemma \ref{lem:supp_ball}, and thank the anonymous referees for their careful reading and suggesting numerous improvements to the work. 

\section{Small quotient subgroups} \label{sec:small_quotient}

\subsection{Topological preliminaries}
We begin by recalling some basic facts about the topology of homeomorphism and diffeomorphism groups and configuration spaces.  For simplicity, we assume all manifolds are connected, although analogous results hold in the disconnected case, provided one takes into account the surjections from $\Homeo_c(M)$ or $\Diff^r_c(M)$ to the homeomorphism or diffeomorphism groups of any union of connected components of $M$ obtained by restriction to those components.    When we speak of $\Diff^r_c(M)$, we tacitly assume that $M$ has a smooth structure.  We equip the groups $\Homeo_c(M)$ and $\Diff^r_c(M)$ with the standard $C^0$ and $C^r$ compact-open topologies, respectively.   

It follows from deep work of Edwards--Kirby \cite{EK} and Cernavskii \cite{Cernasvskii} that $\Homeo_c(M)$ is locally connected (the main result of \cite{EK} is that the homeomorphism groups of a compact manifold is locally contractible, and the proof easily gives local connectedness for $\Homeo_c(M)$ for general $M$).  The following is a rephrasing of a simplified version of their major technical theorem \cite[Theorem 5.1]{EK} and a key consequence.  

\begin{prop}[\cite{EK}] \label{prop:EK}
Let $K \subset M$ be compact, $U$ any neighborhood of $K$, and $D \subset K$ a closed (possibly empty) set.  Then any embedding of $K$ into $U$ that is sufficiently close to the identity (i.e. the inclusion) and restricts to the identity on $D$ can be deformed to the identity through embeddings that are identity on $D$, and these embeddings can be taken to have image in $U$.  
\end{prop} 
\begin{cor}[\cite{EK}, Corollary 7.3] \label{cor:stab_connect}
The pointwise stabilizer $\Stab(X)$ of a finite set of points $X$ in $M$ is also a locally connected subset of $\Homeo_c(M)$.  In particular the quotient of $\Stab(X)$ by its identity component is a discrete subgroup.    
\end{cor} 
Local contractibility and hence local connectedness of $\Diff^r_c(M)$ (and the relative version fixing a finite set) is classical, a discussion and references can be found in \cite[Chapter 1]{Banyaga}.  

Another important and well-known consequence of the work of Edwards and Kirby is that $\Homeo_c(M)$ has the {\em fragmentation property}.  
\begin{defn}
A subgroup $G \subset \Homeo(M)$ has the {\em fragmentation property} if, for any open cover of $M$, the group $G$ can be generated by homeomorphisms supported on elements of the cover.  
\end{defn} 
Note that such a group necessarily lies in $\Homeo_c(M)$.  For $\Diff^r_c(M)$, fragmentation is less difficult, and may be proved by splitting up a time-dependent vector field whose time-one flow is the diffeomorphism in question, using a partition of unity. See \cite[Chapter 2]{Banyaga}.

Since configuration spaces will play an important role in this work, we record the following basic tools.  Recall that $\PConf_n(M)$ is defined to be the complement of the fat diagonal in $M^n$, and $\Conf_n(M)$ is its quotient by permutations of the factors.  
\begin{prop}
Let $M$ be a connected manifold and $X \subset M$ be a finite set.  Then $\Homeo_c(M) / \Stab(X)$ is homeomorphic to $\Conf_{|X|}(M)$.  The same holds with $\Homeo$ replaced by $\Diff^r$.  
\end{prop} 

\begin{proof} 
The map $\Homeo_c(M) / \Stab(X) \to \Conf_{|X|}(M)$ given by $[f] \mapsto f(X)$ is bijective and continuous.  We need to show that the inverse of this is continuous.  Suppose we are given a configuration $X_1 \in \Conf_{|X|}(M)$, and given $\epsilon >0$ for some $\epsilon$ less than 1/4 the minimum distance between points in $X_1$, we wish to show that any configuration $\epsilon$-close (pointwise) to $X_1$ can be taken to $X_1$ via a homeomorphism close to the identity.   For each such point $x$, we may find a homeomorphism supported on the $2\epsilon$-neighborhood of $x$ and taking $x$ to its nearby point $y \in X_2$ without moving any of the other points in the configuration.  The composition of all such is supported on a compact set and moves each point distance at most $2\epsilon$, which suffices to prove the claim.   To prove this for diffeomorphisms instead, one may perform a similar proof using the flow of a smooth vector field instead, working in local chart to ensure the flow is close to the identity.   
 \end{proof} 
 
 \begin{prop} \label{prop:cover} 
Let $M$ be a connected manifold and let $X \subset M$ be a finite set.  Then  $\Homeo_c(M)$ is a topological fiber bundle over $\Conf_{|X|}(M) =\Homeo_c(M) / \Stab(X)$, and the same holds with $\Homeo$ replaced by $\Diff^r$.  
\end{prop} 

\begin{proof} 
Since $\Stab(X)$ is a closed subgroup, it suffices to produce a locally defined continuous sections $s: \Conf_{|X|}(M) \to \Homeo_c(M)$, then the local product structure over an open set $U$ where $s$ is defined is given by assigning to $c \in U$ the coset $s(c)\Stab(X)$.  
Without loss of generality, we may do this at the identity, i.e. over the configuration $X$.  Since we are working locally, we may label points as in $\PConf_{|X|}(M)$.  Choose a disk $D_i$ about each $x_i \in X$, small enough so that the $D_i$ are pairwise disjoint.  Let $\lambda_{i, j}$ for $1 \leq j \leq n = \dim(M)$ be smooth vector fields supported on $D_i$ that agree with the coordinate vector fields in a small local chart about $x_i$, and $\phi(t)_{i,j}$ the time $t$ map of the flow of $\lambda_{i,j}$.  For each $i = 1, 2, \ldots |X|$, there is a continuous, injective map $s_i$ from a neighborhood of $0$ in $\R^n$ to $\Diff^\infty_c(M)$ given by $s_i(t_1, \ldots t_n) \mapsto \phi(t_1)_{i,1} \circ \ldots \circ \phi(t_n)_{i,n}$.  This map is a homeomorphism onto its image in $\Diff^\infty_c(M)$.  Thus, any configuration sufficiently close to $X$ can be written uniquely as $s_1(\vec{v}_1) \circ \ldots \circ s_{|X|} (\vec{v}_{|X|})(X)$, (for $\vec{v_i}$ close to $0$ in $\R^n$) which gives the desired local section.  
\end{proof} 

\subsection{Small quotient subgroups} 

\begin{defn}
Let $G$ be a topological group, and $A \subset G$ a subgroup.  We say that $A$ has {\em small quotient} in $G$ or $A$ {\em is a small quotient subgroup} if there exists $n \in \mathbb{N}$ such that, for any continuous, injective map of an $n$-disc $\DD^n \to G$, the projection $\DD^n \to G \to G/A$ is non-injective.  
If $A$ has small quotient, the {\em codimension} of $A$ is the maximum $n$ such that there exists a continuous injective map $\DD^n \to G$ that descends to an injective map to $G/A$.  
\end{defn} 

Before stating our main theorem, we give two basic properties. 

\begin{obs} \label{obs:easy} (Properties of small quotient subgroups)
\begin{enumerate}
\item If $A \subset E \subset G$ are subgroups and $A$ has small quotient, then $E$ also has small quotient, with codimension bounded above by $\codi(A)$.  
\item If $H \subset G$ is a subgroup and $A$ has small quotient in $G$, then $A \cap H$ has small quotient in $H$, and the codimension of $A$ in $G$ is bounded below by the codimension of $A \cap H$ in $H$.  
\end{enumerate} 
\end{obs} 

\begin{proof} 
The first item is just the observation that the projection map $G \to G/E$ factors through $G \to G/A \to G/E$.
The second follows from the fact that $H/(A \cap H)$ embeds in $G/A$.
\end{proof} 

\noindent \textbf{Convention.}  Going forward, in this section $M$ always denotes a connected manifold.   
 
\begin{defn}
A subgroup $G \subset \Homeo(M)$ has {\em local simplicity} if, for any relatively compact open ball $B \subset M$, the subgroup of $G$ consisting of elements with compact support contained in $B$ is nontrivial, path connected and (algebraically) simple. 
\end{defn}

Local simplicity of $\Homeo_c(M)$ is a result of Anderson \cite{Anderson} who shows $\Homeo_c(B)$ is algebraically simple.  The combined work of Epstein \cite{Epstein}, Mather \cite{Mather1, Mather2} and Thurston \cite{Thurston} establishes local simplicity for $\Diff_c^r(M)$ when $1 \leq r \leq \infty$ and $r \neq \di(M)+1$.  Whether $\Diff_c^{\di(B)+1}(B)$ is algebraically simple is a famous open question in the field.  
The first step towards our orbit classification theorem is the following result on closed, small quotient subgroups of locally simple groups.   

\begin{lem}  \label{lem:supp_ball}
Let $G \subset \Homeo(M)$ be a locally simple group, and $A \subset G$ a closed subgroup with small quotient in $G$.  
Then there exists a ball $B \subset M$ such that $A$ contains all homeomorphisms in $G$ supported on $B$.  
\end{lem} 

\begin{proof}  
Let $n=\codi(A)+1$ and fix $n$ disjoint closed balls in $M$.  Let $G_i \subset G$ denote the subgroup of homeomorphisms with support in the $i$th ball.  
Since $G_i$ and $G_j$ commute whenever $i \neq j$, we may identify the product $G' := G_1 \times \ldots \times G_n$ with a subgroup of $G$.  Let $A' = A \cap G'$, this is also a closed subgroup.  
Let $p_i: A' \to G_i$ be the natural projection, and let $A_i = p_i(A')$.    Our goal is to show that some $G_i$ is contained in $A$.  

Observation \ref{obs:easy} (1) implies that $A_1 \times \ldots \times A_n \supset A'$ has codimension less than $n$ in $G'$.  Let $\overline{A_i}$ denote the closure of $A_i$ in $G_i$. Then $\overline{A_1} \times \ldots \times \overline{A_n}$ has codimension less than $n$ also, and 
\[ G'/(\overline{A_1} \times \ldots \times \overline{A_n}) = G_1/\overline{A_1} \times \ldots \times G_n/\overline{A_n}. \]
Suppose that $\overline{A_i} \neq G_i$ for some $i$.  Then $G_i/\overline{A_i}$ contains at least two points, and is Hausdorff (since $\overline{A_i}$ is closed) and path connected; since it is Hausdorff there is an injective path $[0,1] \to G_i/\overline{A_i}$. 
Thus, if $G_i\neq \overline{A_i}$ for {\em all} $i$, we would have an embedded $\DD^n$ in $G_1/\overline{A_1} \times \ldots \times G_n/\overline{A_n}$, contradicting the bound on codimension. 
We conclude that  $G_i = \overline{A_i}$ for some $i$.  Reindexing if needed, we assume $i=1$.  

Let $p: A' \to G_2 \times \ldots \times G_n$ denote the product map $p_2 \times \ldots \times p_n$, and let $K = \Ker(p)$.  Note that $K = A \cap G_1$, which is a closed subgroup of $G_1$.  Since $A$ has small quotient in $G$, this means that $K = A \cap G_1$ has small quotient in $G_1$, so is nontrivial. (Recall that $G_1$ is nontrivial by definition of local simplicity.)  
We have that $p_1(K) = K$ and $p_1(A') = A_1$.  Since the image of a normal subgroup under a surjective group homomorphism is normal, $K = p_1(K)$ is a nontrivial normal subgroup of $p_1(A') = A_1$.  In other words, $A_1$ is contained in the normalizer $N_{G_1}(K)$ of $K$ in $G_1$.  
It is a basic fact that the normalizer of a closed subgroup of a topological group is closed (if $H \subset G$ is closed, then $N_G(H) = \bigcap_{h \in H} \{g \in G \mid ghg^{-1} \in H\}$ is an intersection of closed sets because $g \mapsto ghg^{-1}$ is continuous), so $N_{G_1}(K)$ is a closed subgroup containing $A_1$, i.e. equal to $G_1$.  
By local simplicity, $G_1 = K = A \cap G_1$, which is what we needed to show.  
\end{proof}

Borrowing terminology from \cite{GM}, we make the following definition.
\begin{defn}
A subgroup $G \subset \Homeo(M)$ is {\em locally continuously transitive} if for all $x, y \in M$ and ball $B$ containing $x$ and $y$, there exists a 1-parameter subgroup $f_t$ of $G$ supported on $B$  such that $f_1(y) =x$. 
\end{defn}

\begin{lem} \label{lem:nontransitive}
Suppose $G \subset \Homeo(M)$ is a locally continuously transitive group 
and $A \subset G$ a closed subgroup with small quotient in $G$.  
Then there exists a finite set $X$ such that $A$ setwise preserves $X$ and acts transitively on $M - X$; or in the case $\dim(M) = 1$, acts transitively on connected components of $M-X$. 
\end{lem}

\begin{proof} 
If the action of $A$ is transitive, the statement is immediate with $X = \emptyset$.  Thus we may assume the action of $A$ is nontransitive, so there are points $x, y \in M$ such that $y \notin A\cdot x$.  Let $\{ f_t\}_{t \in \R} \subset G$ be a 1-parameter subgroup of $G$ consisting of homeomorphisms supported on a precomact neighborhood of $x$ such that $f_1(y) = x$.  In particular, this flow $f_t$ is not a subgroup of $A$.  Since $A$ is closed, $\{t : f_t \in A\}$ is a closed, proper subgroup of $\R$, so either trivial or isomorphic to $\Z$.  In the first case $\{f_t\}/ (\{f_t\} \cap A)$ is $\R$, and in the second case it is $S^1$; either contains an embedded $1$-dimensional disc. 

More generally, suppose we can find distinct points $x_i$ and $y_i$, $1 \leq i \leq m$, where $y_i \notin A \cdot x_i$.  In the case where $M$ is 1-dimensional, we further require that $y_1 < x_1 < ... <y_m < x_m$ with respect to the ordering induced by $\R$ in some local coordinate chart. Then we can take a continuous homomorphism $\R^n \to G$, where the $i$th factor is a flow whose time 1 map takes $y_i$ to $x_i$, with support disjoint from the other factors.  This can be constructed iteratively, as follows. Take a simple path from $y_1$ to $x_1$ disjoint from each of the remaining points $x_i, y_i$ and define a vector field tangent to the path and supported in a small neighborhood $U_1$ of the path, again disjoint from the other points.  An appropriate time scaling of the flow of this vector field will suffice.  If $U_1$ is chosen small enough, $M - U_1$ will be connected (or, in the 1-dimensional case, have each remaining pair $x_i$ and $y_i$ in the same connected component), and the process may be repeated.

Then $\R^n \cap A$ is a closed subgroup of $\R^n$, and each element of $\R^{i-1} \times \{1\} \times \R^{n-i}$ takes $y_i$ to $x_i$, so is not an element of $A$. We know that every linear subspace of $\R^n$ has to intersect one of $\R^{i-1} \times \{1\} \times \R^{n-i}$ because at least one coordinate should be nonzero. Therefore $\R^n \cap A\subset \R^n$ does not contain any linear subspaces, which means that $\R^n \cap A\subset \R^n$ is a discrete subgroup. Thus, $\R^n / (\R^n \cap A)$ is $n$-dimensional; and if $D$ is a small topologically embedded $n$-disc in $\R^n$, then the projection of $D$ to the quotient $G/A$ will be injective.  

Thus, if $A$ has codimension at most $n$, then the maximal value of $m$ for which we may find such a set of points is $m = n$.  Fix such a maximal collection of points $x_1, \ldots, x_m, y_1,...,y_m$.   In the case where $\di(M) \geq 2$, maximality implies that $M - \{x_1, \ldots, y_m\}$ is contained in a single orbit $\mathcal{O}$. Set $X = M - \mathcal{O}$, which is a subset of $\{x_1, \ldots, y_m\}$, hence finite. By definition of orbit, $A$ preserves $\mathcal{O}$ and the complement $X$, which is what we needed to show.  

In the case $\di(M) = 1$, maximality implies that each connected component of $M - \{x_1, \ldots, y_m\}$ is contained in a single orbit, and we may replace $\mathcal{O}$ with the union of these finitely many orbits and conclude as above.  
\end{proof}

\begin{defn}
For a finite set $X$, let $\Stab(X)$ denote the setwise stabilizer of $X$, and $\Stab(X)_0$ the connected component of $\Stab(X)$ containing the identity.  
\end{defn} 

\begin{prop}\label{Stab0}
The group $\Stab(X)_0$ is path-connected, and $\overline{\Homeo_c(M-X)} \subset \Stab(X)_0$.
\end{prop}

\begin{proof}
Path-connectedness of $\Stab(X)_0$ is Corollary \ref{cor:stab_connect}.   The containment $\overline{\Homeo_c(M-X)} \subset \Stab(X)_0$ is clear, since $\Homeo_c(M-X)$ by definition contains homeomorphisms isotopic to the identity relative to $X$, and $\Stab(X)_0$ is closed.   For the reverse containment, suppose one is given a path $f_t$ with $f_0 = id$ in $\Stab(X)_0$.  We need to approximate $f_1$ arbitrarily well by a homeomorphism $\hat{f}_1$ where $\hat{f}_t$ is a path based at $id$ of compactly supported homeomorphisms of $M-X$.  We can do this using techniques of \cite{EK}, via the following argument.   

Given some $\epsilon>0$, choose $\delta$ small enough so that the image of a $2\delta$-neighborhood of $X$ under $f_t$ remains inside the $\epsilon$ neighborhood $N_\epsilon$ of $X$.  Thus, $f_t(M - N_\epsilon) \cap N_{2\delta} = \emptyset$.  Break the path $f_t$ into time intervals $0 = t_0, t_1, \ldots t_k = 1$ small enough so that  $f_t \circ f_{t-1}^{-1}$ is close enough to the identity (considered as an embedding of $M - N_\epsilon$ into $M$) so that Proposition \ref{prop:EK} applies, taking $U$ to be the complement of $\overline{N_\delta}$ in $M$.  Using the proposition, $f_t \circ f_{t-1}^{-1}$ can be isotoped to the identity via an isotopy supported outside of $N_\delta$, which can be extended to a homeomorphism of $M$ that pointwise fixes $N_\delta$.  Composing these isotopies, one produces a path of homeomorphisms whose time one map agrees with $f_1$ on $M - N_\epsilon$ and pointwise fixes $N_\delta$, hence is compactly supported in $M-X$.  
\end{proof}

Using the results above, we obtain the following.  

\begin{thm}  \label{thm:small_quotient_homeo}
Let $M$ be a manifold, and $A \subset \Homeo_c(M)$ a closed subgroup with small quotient.  Then there is a finite set $X \subset M$ such that $\Stab(X)_0 \subset A \subset \Stab(X)$, and $\Homeo_c(M)/A$ is homeomorphic to an intermediate cover of $C_X=\Homeo_c(M)/\Stab(X)_0\to \Conf_{|X|}(M)=\Homeo_c(M)/\Stab(X)$.  
\end{thm}

\begin{proof} 
Since $\Homeo_c(M)$ satisfies the hypothesis of Lemma \ref{lem:nontransitive}, we know that $A$ is contained in $\Stab(X)$ for some finite set $X$, and acts transitively on the complement of $X$.     Let $A'$ denote the intersection of $A$ with $\Homeo_c(M-X) \subset \Homeo_c(M)$.  Then $A'$ is small quotient in $\Homeo_c(M-X)$ by Observation \ref{obs:easy}.  By Lemma \ref{lem:supp_ball}, there exists a ball $B \subset M - X$ such that $A'$ contains all homeomorphisms supported on $B$, hence $A$ contains all homeomorphisms supported on $B$.  
Since $\Homeo_c(M-X)$ has the fragmentation property and the action of $A$ on $M-X$ is transitive, we conclude that $A' = \Homeo_c(M-X)$.  Since $A$ was assumed closed in $\Homeo_c(M)$, it therefore contains the closure of $\Homeo_c(M-X)$, therefore by Proposition \ref{Stab0}, we know that $A$ contains $\Stab(X)_0$.

The map $\Homeo_c(M)/A \to \Homeo_c(M)/\Stab(X)$ has a local section that can be obtained by composing a local section for $\Homeo_c(M) \to \Homeo_c(M)/\Stab(X)$ with the projection to $\Homeo_c(M)/A$.  Since $\Stab(X)/\Stab(X)_0$ is discrete, $A/\Stab(X)_0$ is discrete, so $\Homeo_c(M)/A$ is a cover of $\Homeo_c(M)/\Stab(X)$, as claimed.  
\end{proof}

\para{Diffeomorphism case}
Our next goal is to prove the counterpart to Theorem \ref{thm:small_quotient_homeo} in the diffeomorphism case.
Here there are more small quotient subgroups -- for example the group of diffeomorphisms fixing a point with trivial first derivatives at that point also has finite codimension.  In general, we will see that one needs to consider $r$-jets at finite sets of points.  

 We start by recalling the notion of jet spaces.  For a smooth manifold $M$, an {\em r-jet} of a map $M \to M$ is an equivalence class of triples $(x, f, U)$, where $f: U \to M$ a $C^r$ map, $U$ is an open neighborhood of $x$, and $(x, f, U)$ is equivalent to $(x, g, U')$ if all derivatives at $x$ up to order $r$ of $f$ and $g$ agree.  The {\em space of $r$-jets of $C^r$ maps of $M$}, denoted $J^r(M, M)$, is a fiber bundle over $M \times M$, via the natural projection map assigning $(x, f, U)$ to $(x, f(x))$, with linear structure group.  

We will be interested in a related bundle where $M$ is a configuration space.  
\begin{defn} 
For a smooth manifold $M$, and finite set $X \subset M$, let $J^r(\Conf_n(M))$ denote the pullback of the bundle $J^r(\Conf_n(M), \Conf_n(M))$ under the diagonal map $\Conf_n(M) \to \Conf_n(M) \times \Conf_n(M)$.   We call this the {\em configuration r-jet bundle}.  
\end{defn} 

Since $\Diff^r(M)$ is naturally a subgroup of $\Diff^r(\Conf_n(M))$, there is a natural action of $\Diff^r(M)$ on $J^r(\Conf_n(M))$ by bundle automorphisms.  
We set notation for a point stabilizer of this action.  

\begin{defn} 
Let $\Stab^r(X) \subset \Stab(X) \subset \Diff^r(M)$ denote the point stabilizer of the equivalence class of the identity map at $X$ in $J^r(\Conf_{|X|}(M))$ under the natural action of $ \Diff^r(M)$ on $J^r(\Conf_{|X|}(M))$.  
We let $\Stab^r(X)_0$ denote the identity component of $\Stab^r(X)$.  
\end{defn} 
One should think of $\Stab(X)/\Stab^r(X)$ as the ``$r$-jets of diffeomorphisms at $X$'', we call the larger group $\Stab(X)/\Stab^r(X)_0$ an {\em extended jet group at $X$}.  
By definition, we have $J^r(\Conf_n(M)) = \Diff^r(M)/\Stab^r(X)$.  

Note also that, if $f \in \Stab^r(X)_0$, then $f$ pointwise fixes $X$, has all derivatives up to order $r$ agreeing with the identity map at each point of $X$, and is isotopic to the identity through a path of such maps.  Conversely, if a map has identity $r$-jet at each point of $X$ and is isotopic to the identity through a path of such maps, then it clearly lies in $\Stab^r(X)_0$.  Using this, we prove the following.  
 

\begin{lem} \label{lem:closure}
Let $r \geq 1$.  Then the closure of $\Diff^r_c(M-X)$ in $\Diff^r_c(M)$ is 
$\Stab^r(X)_0$.
\end{lem}

\begin{proof} 
By definition, we have $\Diff^r_c(M-X) \subset \Stab(X)^r_0$.  Since the latter group is closed, this gives $\overline{\Diff^r_c(M-X)} \subset 
\Stab^r(X)_0$. For the other direction, as we did for homeomorphisms, here one needs to approximate the endpoint a path of diffeomorphisms agreeing up to order $r$ with the identity at $X$ using one consisting of diffeomorphisms supported away from $X$.  This can be done by the standard ``blow up" construction, in which one replaces each point of $X$ with its sphere of tangent directions, giving a manifold with sphere boundary components, then gluing in balls to recover a manifold diffeomorphic to M (via a diffeomorphism close to the identity and identifying the center of each ball with a point of $X$).  There is a natural, continuous extension of $C^r$ diffeomorphisms of $M$ preserving $X$ to $C^{r-1}$ diffeomorphisms of the blow-up (see \cite[Theorem 6.1.]{Stark} for the statement before compactifying by spheres, one simply extends the linear diffeomorphisms on the sphere boundaries over the glued-in balls), the action being trivial on the tangent space to $X$ results in a diffeomorphism compactly supported away from $X$.  One may compensate for the loss of regularity using the fact that $C^{r-1}$ diffeomorphisms can be approximated by $C^r$ diffeomorphism. 
\end{proof} 




\begin{thm} \label{thm:small_quotient_diff} 
If $A \subset \Diff^r_c(M)$ is a closed subgroup with small quotient (for some $1 \leq r \leq \infty$) then 
$\Stab^r(X)_0 \subset A\subset \Stab(X)$ for a finite set $X\subset M$, and $\Diff^r_c(M)/A$ has the structure of a cover of $\Conf_{|X|}^r(M)$ under a quotient by a subgroup of extended jet group.\end{thm}

\begin{proof}
To avoid the problem where $\Diff^r_c(M)$ is not known to be simple when $r = \di(M)+1$, we pass immediately to working with subgroups of smooth diffeomorphisms.  Let $A'$ be the closure in the $C^\infty$ topology of $A \cap \Diff^\infty_c(M)$ in $\Diff^\infty_c(M)$.  Note that $A' \subset A$ since the $C^\infty$ topology is finer than the $C^r$ topology.   By Lemmas \ref{lem:supp_ball}, Lemma \ref{lem:nontransitive} and the fragmentation property for $\Diff^\infty_c(M-X)$, there is a finite set $X$ such that $A' \subset \Stab(X)$ and $A'$ contains $\Diff^\infty_c(M - X)$.  The $C^r$ closure of $\Diff_c^\infty(M-X)$ contains $\Diff^r_c(M-X)$, so using Lemma \ref{lem:closure} we conclude that 
\[ \Stab^r(X)_0 \subset A \subset \Stab(X) \] 



Thus, $\Stab^r(X)_0  = \Stab^r(X)_0 \cap A$ and 
\[ \Diff^r_c(M)/(\Stab^r(X)_0 \cap A) = \Diff^r_c(M)/ \Stab^r(X)_0.\] 
Since $\Stab^r(X)/ \Stab^r(X)_0$ is discrete, $ \Diff^r_c(M)/ \Stab^r(X)_0$ is a cover of the configuration jet bundle $ \Diff^r_c(M)/ \Stab^r(X)$. 
The map $\Diff^r_c(M)/(\Stab^r(X)_0 \cap A) \to \Diff/A$ is a quotient by the action of $A/ \Stab^r(X)_0  \subset \Stab(X)/ \Stab^r(X)_0$ (noting that $\Stab^r(X)_0$ is normal in $A$).    

\end{proof}

\section{Orbit classification theorem} \label{sec:ac}

Building on the work in the previous section, we now prove Theorems \ref{thm:orbit_classification} and \ref{thm:lifting}.  We need the following version of the classical invariance of domain theorem. 
\begin{lem}[Invariance of domain for a finite-dimensional CW complex]\label{Invariance}
Let $P$ be an $n$-dimensional CW complex.  Then there is no injective continuous map $\mathbb{R}^{n+1} \to P$.
\end{lem}
\begin{proof}
Assume for contradiction that there is such an embedding $f: \mathbb{R}^{n+1}\to P$. Let $D$ be a closed disk in $\mathbb{R}^{n+1}$ and $B$ be the interior of $D$. The image $f(D)\subset P$ is compact so it only intersects finitely many cells. Find $x\in f(B)$ such that $x$ lands on the maximal dimensional cell $C(x)$ among the cells that intersect $f(D)$. By assumption on the maximal dimension, we know that $f^{-1}(C(x))$ is open in $\mathbb{R}^{n+1}$. This shows that there is an injective, continuous image of $\mathbb{R}^{n+1}$ inside $C(x)$, which contradicts the well known invariance of domain theorem for Euclidean spaces.
\end{proof}

The other ingredients we will need are a collection of automatic continuity results.   Following \cite{Hurtado}, call a homomorphism $\phi$ from $\Diff^r_c(M)$ or $\Homeo_c(M)$ to a topolgoical group $G$ {\em weakly continuous} if, for every compact set $K \subset M$, the restriction of $\phi$ to the subgroup of homeomorphisms supported on $K$ is continuous. 

\begin{thm}[Hurtado \cite{Hurtado}]
Let $M$ and $N$ be smooth manifolds.  Any homomorphism $\phi: \Diff^\infty_c(M) \to \Diff^\infty_c(N)$ is weakly continuous. 
\end{thm} 
\begin{thm}[Mann \cite{Mann}, Rosendal \cite{Rosendal},  Rosendal--Solecki \cite{RS}] \label{thm:ac_homeo}
Let $M$ be a topological manifold and $G$ a separable topological group.  Any homomorphism $\phi: \Homeo_c(M) \to G$ is weakly continuous. 
\end{thm}

If $M$ itself is compact, weak continuity is of course equivalent to continuity, but this is false in general.  For instance (as remarked in \cite{Hurtado}) if $h : \R^n \to \R^n$ is an embedding with image the open unit ball $B$, then conjugation by $h$ produces a homomorphism from $\Diff_c^\infty(\R^n)$ to $\Diff_c(B)$, and extending diffeomorphisms in the image to act by the identity outside of $B$ gives a discontinuous homomorphism  $\Diff_c^\infty(\R^n) \to \Diff_c^\infty(\R^n)$.    Nevertheless, weak continuity combined with our orbit theorem will suffice for our intended applications.

\begin{proof}[Proof of Theorem \ref{thm:orbit_classification}]
We begin with the homeomorphism group case, the diffeomorphism case being analogous. 

\para{Homeo case} Recall as usual that $M$ denotes a manifold without boundary.  Suppose that $\Homeo_c(M)$ acts nontrivially on a finite-dimensional CW complex $N$.  
 If $M$ is compact, then Theorem \ref{thm:ac_homeo} implies the action is continuous, so for any $x \in N$, the stabilizer $G_x$ of $x$ under the action is a closed subgroup of $\Homeo_c(M)$, and the orbit of $x$ gives a continuous, injective map from $\Homeo_c(M)/G_x$ into $N$.  By Lemma \ref{Invariance}, the stabilizer $G_x$ is a small quotient subgroup (of codimension at most $\di(N)$) so by Theorem \ref{thm:small_quotient_homeo}, $\Homeo_c(M)/G_x$ is homeomorphic to a cover of $\Conf_n(M)$ for some $n$.  
 
If $M$ is not compact, we need to employ an additional argument.  
To simplify notation, let $A = G_x$ be the point stabilizer of the action on $N$. The goal is again to show that there is a finite set $X\subset M$ such that 
\[
\Stab(X)_0 \subset A\subset \Stab(X),
\]
which implies that $A$ is a closed subgroup (since it is a union of components of $\Stab(X)$) and that the orbit space is a cover of $\Conf_{|X|}(M)$.  

Let $K\subset M$ be a connected, compact set that is the closure of a connected, open subset $\mathring{K}$ of $M$. Let $\Homeo_K(M) \cong \Homeo_c(K)$ be the subset of $\Homeo_0(M)$ consisting of elements with support in $\mathring{K}$ and isotopic to identity by an isotopy compactly supported in $\mathring{K}$. By weak continuity, the restriction of the action to $\Homeo_K(M)$ is continuous.  The group $A_K:=A\cap \Homeo_K(M)$, is a closed, small quotient subgroup of $\Homeo_K(M)$.  For a finite set $X \subset K$, let $\Stab_K(X)$ denote the stabilizer of $X$ in $\Homeo_K(M)$, and $\Stab_K(X)_0$ its identity component. By Theorem \ref{thm:small_quotient_homeo}, there exists a finite set $X(K) \subset K$ such that 
\[
\Stab_K(X(K))_0\subset A_K\subset \Stab_K(X(K)).
\]
Also, the cardinality of $X(K)$ multiplied by the dimension of $M$ is the dimension of the orbit of $x$ in $N$ under this subgroup, and this is bounded above by the dimension of $N$.  In particular, $|X(K)| \leq \di(N)$.  

We claim that whenever $K\subset H$ are both the compact closures of connected sets as above, we have $X(K)\subset X(H)$.  To show this, observe that $A_{H}\cap \Homeo_K(M) = A_{K}$, and we know that 
\[\Stab_H(X(H))_0\cap \Homeo_K(M) \subset A_H\cap\Homeo_K(M) =A_K \subset \Stab_K(X(K)).\]
However $\Stab_K(X(H)\cap K)_0\subset \Stab_H(X(H))_0\cap \Homeo_K(M)\subset \Stab_K(X(K))$, which implies that $X(K)\subset X(H)\cap K\subset X(H)$ because $\Stab_K(X(H)\cap K)_0$ acts on $\Int(K)-X(H)\cap K$ transitively.

Let $K_n$ be an exhaustion of $M$ by connected compact sets with $K_n \subset K_{n+1}$, each one the closure of a connected open set.  Then $X(K_n) \subset X(K_{n+1})$, and the fact that the cardinality of $X(K_n)$ is bounded by $\dim(N)$ implies that this sequence is eventually constant, equal to some finite set $X$.  Without loss of generality, we may modify our compact exhaustion so that $X(K_n) = X$ for all of the compact sets $K_n$.  
The above discussion shows that for all $K_n$, 
\[
\Stab_{K_n}(X)_0\subset A\cap \Homeo_{K_n}(M) \subset \Stab_{K_n}(X)
\]
This implies that $A\subset \Stab(X)$ since $A \subset \Homeo_c(M) = \bigcup_n \Homeo_{K_n}(M)$. On the other hand, we also have that $\Stab(X)_0=\bigcup \Stab_{K_n}(X)_0$ because the union of supports of elements in a path in $\Stab(X)_0$ is also compact, so we have $\Stab(X)_0 \subset A$, which is what we needed to show.  
\qed

\para{Diffeo case} 
The same argument above applies for any {\em continuous} action of $\Diff^r_c(M)$ on $N$ by homeomorphisms, and we conclude that the orbit is the continuous, injective image of a cover of one of the spaces given in Theorem \ref{thm:small_quotient_diff}.    If $M$ is compact, $r = \infty$, and the action is by smooth diffeomorphisms, then continuity follows immediately from Hurtado's theorem.  In the noncompact case for a weakly continuous action, one simply repeats the argument above using an exhaustion by compact sets.  
 \end{proof}

We can now give Theorem \ref{thm:transitive} as an elementary consequence. 
\begin{proof}[Proof of Theorem \ref{thm:transitive}]
Suppose $\Homeo_c(M)$ acts transitively on a manifold or  finite dimensional CW complex $N$, thus there is only a single orbit, and hence by the Orbit Classification Theorem, a continuous, bijective map from some cover $C$ of some configuration space $\Conf_n(M)$ to $N$.  Thus $\di(N) \geq \di(M)$, and if equality holds the orbit map is a homeomorphism.   To show equality holds (ruling out pathological behavior like a space filling curve), consider the restriction of the orbit map to some compact subset $K$ of $C$. This restriction is a homeomorphism onto its image, a compact subset of $N$; since $C$ can be exhausted by a countable collection of compact sets, by Baire category the image of $K$ must be somewhere dense, hence contain a ball and thus $\di(N) = \di(M)$.   
\end{proof}

As another easy consequence, we can prove Theorem \ref{thm:lifting}.

\begin{proof} 
Suppose $M$ and $N$ are manifolds with $\di(M) \geq \di(N)$, and we have a nontrivial action of $\Homeo_c(M)$ on $N$.  
Take any point $x \in N$ not globally fixed by the action, and let $G_x$ denote its stabilizer.  Then the orbit of $x$ gives a continuous, injective map of $\Homeo_c(M)/G_x$ into $N$.  By Theorem \ref{thm:small_quotient_homeo} and our assumption on dimension, the space $\Homeo_c(M)/G_x$ is a covering $M_x$ of $\Conf_1(M) \cong M$.   In particular, we must have $\di(M) \leq \di(N)$, hence equality holds.  Since $N$ is second countable, and orbits are disjoint, there can be at most countably many such disjoint embedded copies of covers of $M$, which proves the theorem.  
\end{proof} 

The same proof applies to any weakly continuous action of $\Diff^r_c(M)$ on $N$ by homeomorphisms.  In Section \ref{sec:diffeo_structure} we will improve this result for the case of actions $\Diff^r_c(M) \to \Diff^s(N)$ where $s \geq 1$, showing that in this case $s \leq r$ and $N$ must itself necessarily be a cover of $M$.

\section{Admissible covers: classifying orbit types} \label{sec:E_subgroup}
We now describe which spaces occur as orbits for actions of $\Homeo_c(M)$; equivalently, we classify the covers of $\Conf_n(M)$ that admit transitive actions of $\Homeo_c(M)$.  We keep the notation from the previous sections.  If $X \subset M$ is a finite set, there is a fiber bundle $\Stab(X) \to \Homeo_c(M) \to \Conf_{|X|}(M)$. From the long exact sequence of homotopy groups, we have the following exact sequence
\begin{equation} \label{eq:push}
\pi_1(\Homeo_c(M)) \xrightarrow{ev} \pi_1(\Conf_{|X|}(M))\xrightarrow{p} \pi_0(\Stab(X))\to 1
\end{equation}
where the evaluation map $ev$ takes a loop $f_t$ in $\Homeo_0(M)$ based at the identity to the path of configurations $f_t(X)$.  The map $p$ is analogous to the familiar ``point push" map from the Birman exact sequence in the study of mapping class groups of surfaces, in this special case, $X$ is a singleton, and a loop based at $X$ is sent to the isotopy class of a homeomorphism obtained by pushing the point around the loop.  
Let $E_{|X|}$ denote the image of $ev$.   From the long exact sequence of homotopy groups, we know that 
\[
\pi_1(\Conf_{|X|}(M))/E_{|X|}\cong \Stab(X)/\Stab(X)_0\]
and hence $\Homeo_c(M)/\Stab(X)_0  \to \Homeo_c(M)/\Stab(X) =  \Conf_{|X|}(M)$ is a covering map with deck group $\pi_1(\Conf_{|X|}(M))/E_{|X|}\cong \Stab(X)/\Stab(X)_0$.   

\begin{defn}
Let $C_X = \Homeo_c(M)/\Stab(X)_0$.   We call the covering space $\pi_X:C_X\to \Conf_{|X|}(M)$ the \emph{maximal admissible cover}, and call a cover $Z$ \emph{admissible} if it is an intermediate cover, i.e. it satisfies $C_X \to Z \to \Conf_{|X|}(M)$.
 \end{defn} 
\begin{prop}\label{lift}
Every admissible cover admits a transitive, continuous action of  $\Homeo_c(M)$.  Conversely, every orbit of a continuous action of $\Homeo_c(M)$ on another manifold is the image of an admissible cover under a continuous, injective map. 
\end{prop} 

\begin{proof}
An admissible cover $C_X\to Z\to \Conf_n(M)$ corresponds to a subgroup $A$ satisfying $\Stab(X)_0\subset A\subset \Stab(X)$. Therefore $\Homeo_c(M)$ naturally acts transitively (by left-multiplication) on $Z = \Homeo_c(M)/A$. 


To show the second half of the statement, if $Y$ is some orbit of an action of $\Homeo_c(M)$ and $G_y$ is the point stabilizer for some $y \in Y$, then $\Homeo_c(M)/G_y$ injects continuously onto $Y$, and by Theorem \ref{thm:small_quotient_homeo} we have $\Stab(X)_0\subset G_y \subset \Stab(X)$, so $\Homeo_c(M)/G_y$ is an admissible cover. 
\end{proof}

\para{Classification of admissible covers}
By Proposition \ref{lift}, classifying admissible covers reduces to the problem of            understanding the image of $ev$, which of course depends on the topology of $M$.   
In general, the image $E_X$ of $ev$ is contained in $\pi_1(\PConf_{|X|}(M)) \subset \pi_1(\Conf_{|X|}(M))$.  It is also a {\em central} subgroup of  $\pi_1(\Conf_{|X|}(M))$, since if $\gamma$ is a loop in $\pi_1(\Conf_n(M))$ based at $X = \{x_1, \ldots x_n\}$, and $f_t$ a loop in $\Homeo_0(M)$, then $f_t(\gamma)$ gives an isotopy between $\gamma$ and $ev(f_t) \gamma ev(f_t)^{-1}$, so these two elements of $\pi_1(\Conf_n(M))$ agree.   (See also \cite[Lemma 3.10]{Hurtado}.)

While it is not clear whether any general classification results hold beyond this observation on the center, 
in low dimensions it is possible to give a complete description of which admissible covers may occur:  

\begin{itemize}
\item \textbf{Dimension 1.}  It is an easy exercise to see that $\pi_1(\Homeo_0(S^1))=\mathbb{Z}$ and also $\pi_1(\PConf_n(S^1))=\mathbb{Z}$. 
(Recall that by convention $\PConf_n(S^1)$ denotes the configuration space of $n$ cyclically ordered points on $S^1$.).  
Thus, all admissible covers are intermediate covers of the covering $\PConf_n(S^1)\to \Conf_n(S^1)$. 

\item \textbf{Higher genus surfaces.} 
For a surface $S_g$ of genus $g>1$, Hamstrom \cite{Hamstrom} showed that $\Homeo_0(S_g)$ is contractible.  Therefore $E_X=1$ for any $X$ and all covers of $\Conf_n(S_g)$ are admissible. 

\item \textbf{2-sphere.}  By Kneser \cite{Kneser} we have $\pi_1(\Homeo_0(S^2))\cong \pi_1(SO(3))=\mathbb{Z}/2$.
For $n \geq 3$, the center of $\pi_1(\Conf_n(S^2))$ is $\mathbb{Z}/2$ (see \cite[\S 9.1]{Primer}), which is exactly the image of $ev$.  

\item \textbf{2-torus.} We have $\pi_1(\Homeo_0(T^2))\cong \pi_1(T^2)=\mathbb{Z}^2$ (see \cite{Hamstrom2}).  As remarked above, the induced map $\pi_1(T^2)\to \pi_1(\Homeo_0(T^2))\overset{ev} \to \pi_1(\PConf_n(T^2))$ has image in the center of $\pi_1(\Conf_n(T^2))$.  It is injective, since composition with the forgetful map (forgetting all but one point of $x$) gives a map $\pi_1(\PConf_n(T^2)) \to \pi_1(\PConf_1(T^2)) = \pi_1(T^2)$ on which this is clearly an isomorphism.   Moreover, since kernel of this map $\pi_1(\PConf_n(T^2))\to \pi_1(T^2)=\mathbb{Z}^2$ is center-free, we can conclude that the image of $ev$ is the center of $\pi_1(\Conf_n(T^2))$.  
\end{itemize} 

The observation that the image of $ev$ is the center of $\pi_1(\Conf_n(T^2))$ discussed above gives a different proof of the following theorem of Birman.  
\begin{cor}[Birman, see \cite{torusbraid} Proposition 4.2 and \cite{Birman}]
The center of $\pi_1(\Conf_n(T^2))$ is $\mathbb{Z}^2$.
\end{cor} 

\para{Higher dimensions}
When $\di(M)\ge 3$, we know that $\pi_1(\PConf_n(M)) \cong \pi_1(M, x_1)\times...\times \pi_1(M, x_n)$ where $X = \{x_1, \ldots x_n\}$. Fixing a path between $x_1$ and $x_i$ gives an identification of $\pi_1(M, x_1)$ with $\pi_1(M, x_i)$, and under this identification, the definition of $ev$ as the evaluation map implies that its image in $\pi_1(\PConf_n(M)) \subset \pi_1(\Conf_n(M))$ lies in the diagonal subgroup 
\[ \{ (g, g, \ldots, g) \in \pi_1(M)\times...\times \pi_1(M) : g \in E_1\}. \]

In special cases, the geometry of $M$ can give some additional insight into the image of $E_1$.  
For instance, if $M$ is a compact manifold admitting a metric of negative curvature, then $\pi_1(M)$ is center-free, so every cover is admissible.  By contrast, if $M$ is a compact Lie group, the left multiplication action gives $M<\Homeo_0(M)$, so the image $E_1$ is everything. Thus, if $\di(M)\ge 3$ and $M$ is a compact Lie group, then $E_n$ is the whole diagonal.

\begin{rem}[Quotients of $\Diff^r_c(M)$]
While the work in this section focused on homeomorphism groups, it is an equally interesting question to classify the finite dimensional spaces which occur as quotients of $\Diff^r_c(M)$ by small quotient closed subgroups.  Theorem \ref{thm:small_quotient_diff} and our work in the Homeo case reduces this to a problem (though a rather nontrivial one) about understanding quotients of jet groups.   We hope to pursue this in future work. 
\end{rem}

\section{Generalized flat bundle structure for actions by homeomorphisms}  \label{sec:bundle}
This section gives the proof of Theorem \ref{flatbundle}.  Recall that this is the statement that for connected manifolds $M$ and $N$ with $\di(N)<2\di(M)$, if there is an action of $\Homeo_c(M)$ on $N$ without global fixed points, then $N$ has the structure of a generalized flat bundle over $M$, and when $\di(N)-\di(M)<3$, the fiber $F$ is a manifold as well.  
Since $\Conf_n(M)$ has dimension $n\di(M)$, the assumption that $\di(N)<2\di(M)$ eliminates the possibility that any orbit is a cover of a configuration space of two or more points.  This is what makes the orbit gluing problem tractable and leads to the relatively simple statement of the theorem.  

\begin{defn} Let  $B,F$ be topological spaces.  A space $E$ is a called a (generalized) flat bundle with base space $B$ and fiber $F$ if there exists a homomorphism $\phi: \pi_1(B)\to \Homeo(F)$ such that \[
E=(\tilde{B}\times F)/\pi_1(B)\]
where the quotient is by the diagonal action of $\pi_1(B)$ via deck transformations on the universal cover $\tilde{B}$ and by $\phi$ on $F$.
\end{defn} 
Each generalized flat bundle is, in particular, a topological $F$-bundle over $B$ via the map $p:E\to B$ induced by projection to the first factor in $\tilde{B}\times F$.   If $B$ and $F$ are topological manifolds, then this is just the usual definition of a topological {\em flat} or {\em foliated} bundle.  If $B$ and $F$ are smooth manifolds and $\phi$ is an action by diffeomorphisms, this agrees with the differential geometric notion of $E$ admitting a flat connection.  We now prove the structure theorem stated in the introduction.

\begin{proof}[Proof of Theorem \ref{flatbundle}]
Let $\rho: \Homeo_c(M)\to \Homeo(N)$ be an action without global fixed points. Since an orbit is a subset of $N$ which has dimension strictly less than $2\di(M)$ (by the assumption in the Theorem statement), no orbit can be a cover of $\Conf_n(M)$ for $n>1$. Therefore every orbit is the image of a cover of $\Conf_1(M) = M$. 

We first define the projection map $p: N \to M$ and fiber $F$. Let $y\in N$ be any point, and let $G_y \subset \Homeo_c(M)$ denote the stabilizer of $y$ under the action of $\rho$.   Then $G_y$ is a closed subgroup of $\Homeo_0(M)$ with small quotient, so there exists a unique point $x\in M$ such that $\Stab(x)_0 \subset G_y$ by Theorem \ref{thm:small_quotient_homeo}.  We set $p(y) = x$.  

If $f \in \Homeo_c(M)$, then $\Stab(f(x))_0 = f \Stab(x)_0 f^{-1} \subset G_{\rho(f)(y)}$, so the map $p$ satisfies the $\rho$-equivariance property
\[ p \rho(f)(y) = f(p(y)) \] 
for all $y \in N$ and $f \in \Homeo_0(M)$.   We now show that $p$ is continuous.  To see this, let $B$ be a small open ball in $M$, and let $H(B)$ denote the group of homeomorphisms supported on $B$ and isotopic to the identity through an isotopy supported on $B$.  Let $\Fix(\rho(H(B))) \subset N$ denote the set of points fixed by every element of $\rho(H(B))$; this is an intersection of closed sets so is closed.    We claim that
\[ 
p^{-1}(B)=N - \Fix(\rho(H(B)))
\]
which is an open set, so showing this claim proves continuity.  To prove the claim, let $y \in p^{-1}(B)$, and take any $f \in H(B)$ such that $fp(y) \neq p(y)$.  Then $y \notin \Fix(\rho(f))$.  Conversely, if $y \notin p^{-1}(B)$ and $f \in H(B)$, then $f \in \Stab(p(y))_0$ so $\rho(f)$ fixes $y$.  This gives the desired equality of sets, and we conclude $p$ is continuous.  

Now fix a base point $b\in M$ and let $F:=p^{-1}(b)$.
Since $\rho(\Stab(b))$ preserves $F$, the action $\rho$ defines a representation $\Stab(b) \to \Homeo(F)$.  
Also, for any $y \in F$, by definition of $F$, we have $\rho(\Stab(b)_0)\subset G_y$ (recall, this is the stabilizer of $y$ under $\rho$), so this representation factors through $\Stab(b)/\Stab(b)_0$, giving a homomorphism
\[
\phi: \Stab(b)/\Stab(b)_0\to \Homeo(F).
\]
Let $E_b \subset \pi_1(M)$ be the image of the evaluation map for $X = \{b\}$ as defined in Section \ref{sec:E_subgroup}.  We are abusing notation here slightly, writing $E_b$ in place of the more cumbersome $E_{\{b\}}$.  Let $\Gamma:= \Stab(b)/\Stab(b)_0\cong \pi_1(M)/E_b$. 
 Let $C_b$ be the associated maximal admissible cover of $M$, and let $\pi: C_b \to M$ be the natural projection.  $\Gamma$ acts on $C_b$ by deck transformations, and on $F$ by $\phi$.  

We claim that 
$N$ is naturally homeomorphic to $(C_b \times F)/\Gamma$, with orbits of the action of $\Homeo_c(M)$ corresponding to the images of the ``horizontal" sets $C_b \times \{y\}$ in the quotient.  
To see this, fix a basepoint $\tilde{b}$ in $\pi^{-1}(b)$.  Let $L: \Homeo_c(M) \to \Homeo(C_b)$ denote the lifted action of $\Homeo_c(M)$ on $C_b$ as in Proposition \ref{lift}.   For $y \in F$ and $m \in C_b$ let $f_{m} \in \Homeo_c(M)$ denote any homeomorphism satisfying  $L(f_{m})(m)=\tilde{b}$, and define a map $(C_b \times F) \to N$ by 
\[(m, y) \mapsto \rho(f_{m})^{-1}(y).\]
This is independent of the choice of homeomorphism $f_{m}$ since if $g_m$ is another such choice, then $g_m =h f_m$ for some $h \in \Stab(b)_0$.   
We claim that this descends to a well-defined map 
 \[ I: (C_b \times F)/\Gamma \to N\]
and that this map is a homeomorphism giving a flat bundle structure to $N$.

To show $I$ is well defined, we need to show that if $a \in \Gamma$, then $I(a(m), \phi(a)(y)) = I(m, y)$.  
Let $\hat{a} \in \Stab(b)$ be a coset representative of $a \in \Gamma = \Stab(b)/\Stab(b)_0$.  Set $f_{a(m)} = \hat{a} \circ f_m$ (recall that $L(f_m)$ commutes with the deck group of $C_b$), then  
we have 
\[I(a(m), \phi(a)(y)) = \rho(f_{m})^{-1}\rho(\hat{a}) ^{-1}(\phi(\hat{a})(y)) = I(m,y)\]
since by definition, $\phi|_{\Stab(b)}$ agrees with $\rho|_{\Stab(b)}$ on $F$.    

To show injectivity of $I$, note that by construction, orbits of $\rho$ are the images of level sets $C_b \times \{y\}$ in the quotient, and if $\rho(f_m)^{-1}(y)=\rho(f_{m'})^{-1}(y')$, for some $m, m' \in C_b$ and $y, y' \in F$, then $f_{m'}^{-1} f_m \in \Stab(b)$.  To simplify notation, let $g=f_{m'}^{-1} f_m$.  Then $\phi(g)(y)=y'$ and $L(g)(m)=m'$. This means that $(m,y)$ is equivalent to $(m',y')$ in the quotient $(C_b \times F)/\Gamma$, so $I$ is injective.    Surjectivity follows from the fact that each orbit intersects $F$ in at least one point, and the lifted action on $C_b$ is transitive.  
Continuity of $I$ and its inverse follows from the fact that, locally, the homeomorphisms $f_m$ can be chosen continuously with respect to $m$.

Now $(C_b \times F)/\Gamma$ is easily seen to be a flat bundle over $M$ with fiber $F$, as we may write it as 
$(\tilde{M} \times F)/ \pi_1(M)$ where the action of $\pi_1(M)$ is by deck transformations on the first factor, and by the composite action $\pi_1(M) \to \pi_1(M)/E_b \to \Homeo(F)$ of $\pi_1(M)$ on $F$.  This gives the desired flat bundle structure on $N$.  

To conclude the proof, we need to assert that $F$ is a manifold in the low codimension case.  Note that 
in general, the fiber $F=p^{-1}(x)$ may not be a manifold.  For a concrete example, Bing's ``dog bone space" is a non-manifold space $F$ such that $F \times \R$ is homeomorphic to $\R^4$.  Then we may take $M = \R^n$, for $n \geq 3$ and a product action of $\Homeo_c(M)$ on $F \times M \cong \R^{n+3}$, for which the fiber will be $F$.  

However, since the product of the fiber $F$ with a ball is a manifold, we can conclude that $F$ is a {\em generalized manifold} or {\em homology manifold} in the sense of \cite[Chapter 8]{Wilder}. In the case where $M$ has codimension $1$ or $2$, all homology manifolds are manifolds (see \cite[Theorem 16.32]{Bredon}) so $F$ is necessarily a manifold.  See \cite[Chapter 8]{Wilder}.  
\end{proof}

Theorem \ref{flatbundle} has an analog for diffeomorphism groups, and we will see that such pathological fibers do not occur in the differentiable setting.  However, before proving this, we use the homeomorphism group version to solve the extension problem and derive some additional consequences.    We return to work with diffeomorphism groups in Section \ref{sec:diffeo_structure}.

\section{Application: extension problems for homeomorphism groups}  \label{sec:extension} 
In this section we discuss the extension problem, as introduced in \cite{Ghys} and further discussed in \cite{KB}, in the case of homeomorphism groups. 
Recall from the introduction that, if $W$ is a compact, connected manifold with $\partial W=M$, there is a natural ``restrict to the boundary" map
\[
\res(W,M): \Homeo_0(W)\to \Homeo_0(M),
\]
which is surjective, and the extension problem asks whether $\res(W,M)$ has a group theoretic section.  Here, and in the proof, we drop the superscript $0$ from $\res^0(W,M)$ since the context of homeomorphism groups is understood.   We always assume that $M$ is connected.  
We will prove the following stronger version of Theorem \ref{coning}.  As before, we let $E_x$ denote the fundamental group associated to a maximal admissible cover for a singleton $\{x\}$.  

\begin{thm}\label{thm:coning_general}
Let $M$ be a connected, closed manifold of dimension at least $2$ and assume that $\pi_1(M)/E_x$ admits no nontrivial action on $[0,1)$.  Then for any compact $W$ with $\partial W = M$, the map $\res(W,M)$ has a section if and only if $M = S^n, W=\mathbb{D}^{n+1}$.  Furthermore, in the case  $M = S^n$ any extension, and in fact {\em any action} of $\Homeo_0(M)$ on $W$, is conjugate to the standard coning. 
\end{thm}
In fact, the proof applies also when $W$ is noncompact, showing that $W$ is either equal to $\mathbb{D}^{n+1}$, or to $M \times [0,1)$ with an action conjugate to the obvious action preserving each leaf $M \times \{x\}$.  
As an immediate consequence, we have the following examples.  

\begin{cor} 
The following manifolds have no section of $\res(W, M)$, for any manifold $W$ with boundary $M$: 
\begin{enumerate}
\item Any manifold $M \neq S^n$ such that the maximal admissible cover of $M$ is a finite cover, such as when $M$ is a compact Lie group (see the discussion at the end of Section 4) .
\item Any manifold $M$ such that $\pi_1(M)$ itself has no nontrivial action on $[0, 1)$, for example when $\pi_1(M)$ is
\begin{enumerate} 
\item an arithmetic lattice of higher $\mathbb{Q}$-rank (Witte-Morris \cite{Witte}), or
\item a group generated by torsion elements, such as the mapping class groups of a surface, a reflection group, etc.
\end{enumerate} 
\end{enumerate} 
\end{cor}

\begin{rem}
In contrast with Theorem \ref{thm:coning_general}, when $M = S^1$ and $W = \mathbb{D}^2$ there are infinitely many different, non-conjugate extensions.  These are discussed and classified in Section \ref{sec:Militon} below. 
\end{rem}  


\begin{proof}[Proof of Theorem \ref{thm:coning_general}]
Assume that $\rho$ is a section for $\res(W,M)$.  
Define \[W'=\text{the connected component of }W-\Fix(\rho(\Homeo_0(M))) \text{ containing $\partial W$}.\] By Theorem \ref{flatbundle}, there is a canonical flat bundle structure $F\to W'\to M$, where $W'$ is foliated by orbits of the action of $\rho(\Homeo_0(M))$. Since $\dim(W) - \dim(M) = 1$, we know that the fiber $F$ is a 1-dimensional manifold.  
Since $W'$ is connected and the bundle has a section (given by $\partial W$), it follows that $F$ is connected, with connected boundary, and therefore equal to $(0,1]$.  Thus, $W'$ is homeomorphic to $(0,1]\times M$.
By hypothesis, the action of $\pi_1(M)$ on $F$ is trivial, so we have natural coordinates in which $\rho$ is the product of the trivial action on 
 $(0,1]$ and the standard action of $\Homeo_0(M)$ on $M$.  Let $em: (0,1]\times M\to W$  denote the embedding with image $W'$ used to give this coordinate identification.  


Fix $x \in M$, and let $r_n \to 0$ be a sequence in $(0,1]$.  Since $W$ is compact, after passing to a subsequence if needed we may assume that $em(r_n,x)$ converges to some point $\alpha \in W$. 
Note that $\alpha$ is necessarily a fixed point of $\rho$, since it lies on the boundary of a connected component of the open set $W - \Fix(\rho)$.  

\begin{claim}[Shrinking property]\label{sameA}
For any closed ball $U$ around $\alpha$ in $W$, there exists $n_0$ such that $em(\{ r_n\} \times M)\subset U$ for all $n>n_0$.
\end{claim}
\begin{proof}[Proof of claim]
The proof uses the continuity of the action and that $M$ is compact. Let $U$ be a closed ball about $\alpha$.  Suppose for contradiction that $U$ does not contain $em(r_n\times M)$ for all large $n$, so there exists a sequence $x_{n_k} \in M$ with $em(r_{n_k} ,x_{n_k})\in W' - U$.  Since $M$ is compact, after passing to a subsequence we may assume that $x_{n_k}$ converges to a point $y\in M$. 
Take some $f\in \Homeo_0(M)$ such that $f(x)=y$.   
Since $x_{n_k} \to y$, we may find a convergent sequence of homeomorphisms $f_k \to \mathrm{id}$ with the property that $f_k(x_{n_k}) = y$, and so $f^{-1}f_k(x_{n_k}) = x$.

Then $\rho(f_k)em(r_{n_k}, x_{n_k}) = em(r_{n_k}, y)$, and 
$\rho(f_k)em(r_{n_k}, x_{n_k}) = \rho(f) em(r_{n_k}, x)$, which converges to $\rho(f) \alpha = \alpha$.  
Since $f_k$ converges to identity as $k\to \infty$, and the action is continuous, we have 
$em(r_{n_k}, x_{n_k})$ also converges to $\rho(f) \alpha = \alpha$, a contradiction.  
\end{proof}

Let $B$ be any open ball around $\alpha$ and let $S$ be the boundary of $B$.  We call a connected, codimension one closed submanifold $X \subset W$ {\em separating} if $W-X$ has two components.  If $X$ is separating, we call the component of $W -X$ containing $\partial W$ the {\em exterior} component $\Ext(X)$, and the other component the {\em interior}, $\Int(X)$.  
In particular, $B=\Int(S)$ and $W-\bar{B}=\Ext(B)$. We also know that $em(r_n\times M)$ is separating and $\Ext(em(r_n\times M))=em((r_n,1]\times M)$, since these lie in a tubular neighborhood of the boundary given by the image of $em([r_n, 1] \times M)$ while $\Int(em(r_n\times M))=W-em([r_n,1]\times M)$. The following easy claim implies that whenever $em(r_n\times M)\subset B$, we have that $\Int(em(r_n\times M))\subset B$.
\begin{claim}\label{int}
If $X,Y$ are disjoint, separating manifolds in a manifold $W$ with boundary, and $Y\subset \Int(X)$, then $\Int(Y)\subset \Int(X)$.
\end{claim}
\begin{proof}
First, $(W - Y) \cap \Ext(X) = \Ext(X)$ which is connected.  Then $\Int(Y) \cap \Ext(X)$ and $\Ext(Y) \cap \Ext(X)$ partition $(W - Y) \cap \Ext(X) = \Ext(X)$ into two connected components, so one of these sets must be empty.  But $\Ext(X) \cap \Ext(Y)$ contains $\partial W$.  Thus, $\Int(Y) \cap \Ext(X) = \emptyset$, so $\Int(Y)\subset \Int(X)$. 
\end{proof}
Summarizing, from Claim \ref{sameA}, we obtain that for any ball $B$ around $\alpha$, there exists $n_0$ such that $em(r_n\times M)\subset B$ for $n>n_0$. By Claim \ref{int}, we know that $\Int(em(r_n\times M))\subset B$ for $n>N$.  Using this, we deduce the following: 

\begin{claim}\label{ball}
$W=W'\cup \{\alpha\}$ and $W$ is the one point compactification of $(0,1]\times M$.
\end{claim}
\begin{proof}
Each space $em(r_n\times M)$ separates $W$ into two components where $em(s\times M)\subset \Int(em(r_n\times M))$ for $s < r_n$ and $em((t,1]\times M)= \Ext(em(t\times M))$. Therefore $\Fix(\rho(\Homeo_0(M))\cap \Ext(em(t\times M)) = \emptyset$. For any ball $B$ around $\alpha$, there exists $n$ such that $\Int(em(r_n\times M))\subset B$ by the shrinking property. Therefore, $\Fix(\rho(\Homeo_0(M))\cap B = \{\alpha\}$,  which shows that $W=W'\cup \{\alpha\}$, and the topology of $\overline{W'}$ agrees with the one-point compactification topology or Alexandroff extension of $W'$ by the shrinking property.

\end{proof}
The first statement of Theorem \ref{thm:coning_general} now follows from the following proposition. 
\begin{prop}\label{onepoint}
The one point compactification of $M\times (0,1]$ is a manifold if and only if $M$ is the sphere.
\end{prop}
The proof of this Proposition is a standard consequence of the Poincar\'e conjecture.  We recall the outline of the argument for completeness.  Suppose that $N$ is the one-point compactifiction $M\times (0,1] \sqcup \{\infty\}$ and is assumed to be a manifold.  Then the local homology groups $H_k(N, N-\infty; \Z)$ are $\Z$ for $k=0$ and $k= \dim(M)+1$ and equal to $0$ otherwise. This implies that $M$ is a homology sphere.  Since $N$ is homeomorphic to the cone on $M$, it is simply connected; since $M$ is homotopic to $N - \{ \infty \}$ it is also simply connected provided that $\dim(M) > 1$ (in the one dimensional case, the Proposition is trivial by the classification of $1$-manifolds). The Poincar\'e conjecture then implies that $M$ is a sphere, giving the result of the proposition.  

To conclude the proof of the theorem, we need to show that $\rho$ is conjugate to the standard coning when $M = S^n$.  However, this conjugacy is already given by $em$ on $M \times [0,1) = W - \{\alpha\}$, and extends over to the one-point compactifications of these spaces, which is a global fixed point for $\rho$ and for the coning action.  
\end{proof}

Since the $2$-dimensional torus has no nontrivial admissible cover, we obtain the following corollary.
\begin{cor}
$\res(H_g,S_g)$ does not have a section when $g=1$.
\end{cor}
\noindent When $g>1$, the group $\pi_1(S_g)/E_x=\pi_1(S_g)$ has many nontrivial actions on $[0,1]$, which makes it hard to analyze. While we do not know whether $\res(H_g,S_g)$ has a section for $g>1$, in the next section we will prove that no surface has a section in the differentiable category, answering Ghys' original question.

Theorem \ref{thm:coning_general} has a generalization, as follows.  

\begin{thm} \label{thm:dim_plus_one}
Let $M,W$ be closed, connected manifolds with $\di(W)=\di(M)+1\ge 3$, and suppose that the deck group $\pi_1(M)/E_x$ 
of a maximal admissible cover for a singleton $\{x\}$  has no nontrivial action on $S^1$.  
There exists a nontrivial action of $\Homeo_0(M)$ on $W$ if and only if either 
\begin{enumerate}
\item $W=M\times S^1$ and the action is trivial on the $S^1$ factor, or 
\item $M=S^n, W=S^{n+1}$, and the action is by doubling the standard coning.
\end{enumerate}
\end{thm}

\begin{proof}[Proof of Theorem \ref{thm:dim_plus_one}]
Assume there exists a nontrivial homomorphism $\rho: \Homeo_0(M)\to \Homeo(W)$.  Let $W'$ be a connected component of $W-\Fix(\rho(\Homeo_0(M)))$.  By Theorem \ref{flatbundle}, $W' = (C_x \times F)/(\pi_1(M)/E_x)$, where $F$ is a one-manifold.  Since the action of $\pi_1(M)/E_x$ on $F$ is trivial, this bundle has a section, since $W'$ is connected the fiber $F$ is connected as well.   
If $\rho$ has no fixed points, then $W = W'$, $F = S^1$ and $W = M \times S^1$.   


Otherwise, as in the previous proof, $W' \cong (0,1) \times M$ and we let $em$ denote its embedding.  Consider a point $\alpha \in \Fix(\rho(\Homeo_0(M)))$ that can be approached by a sequence of points $em(r_n, x_n)$ with $r_n \to 1$ and $x_n$ converging in $M$.  The proof of Claim \ref{sameA} shows  for any ball $B$ around $\alpha$, there exists $n_0$ such that $em(r_n \times M) \subset B$ for all $n$ sufficiently large.

We modify the argument from the previous proof, as follows.   Delete the subset $em( [1/3, 2/3) \times M)$ from $W$, leaving a single boundary component homeomorphic to $M$.  If this manifold is disconnected, we consider only the connected component containing $\alpha$.  Call this connected manifold with boundary $W''$.    For $p>2/3$, each slice $em (\{p\} \times M)$ separates $\alpha$ from $\partial W''$.  
Modifying the previous definition, we say that for a connected, separating, codimension 1 submanifold $X$ of $W''$ such that $\alpha \notin X$, the component of $W'' - X$ containing $\alpha$ is the {\em interior} and the other component is the exterior.   In parallel to Claim \ref{int} we have
\begin{claim}\label{int2}
If $X,Y$ are disjoint, separating manifolds in $W''$  and $Y\subset \Int(X)$, then $\Int(Y)\subset \Int(X)$.
\end{claim}
\begin{proof}
$Y\subset \Int(X)$ implies that $X \subset \Ext(Y)$, so $(W'' - X) \cap \Int(Y) = \Int(Y)$ which is connected.  Then $\Int(Y) \cap \Int(X)$ and $\Int(Y) \cap \Ext(X)$ partition $(W - X) \cap \Int(Y) = \Int(Y)$ into two connected components, so one of these sets must be empty.  But $\Int(X) \cap \Int(Y)$ contains $\alpha$.  Thus, $\Int(Y) \cap \Ext(X) = \emptyset$, so $\Int(Y)\subset \Int(X)$. 
\end{proof}
We conclude as before that $W''$ is the one point compactification of $em([2/3, 1) \times M$, that $M$ is a sphere, and that $W''$ was one of two connected components of $W - em( [1/3, 2/3) \times M)$.  The same argument applies to the other connected component, and we conclude that the action on $W$ is the double of the standard coning.

%
\end{proof}



\section{Application: bundle structure and the extension problem in the differentiable case} \label{sec:diffeo_structure}
In this section, we prove the structure theorem and then discuss the extension problem for diffeomorphism groups.  
We recall the statement here.  

\begin{thm_bundle}  
Suppose $M$ is a connected, closed, smooth manifold and $N$ is a connected manifold with $\dim(N) < 2\dim(M)$.  If there exists a nontrivial continuous action $\Diff^r_0(M)\to \Diff^s_0(N)$, $0 \leq r \leq \infty$, $1\leq s \leq \infty$, then the action is fixed point free, and $N$ is a topological fiber bundle over $M$ where the fibers are $C^s$-submanifolds of $N$.   
\end{thm_bundle} 

The proof in the case where $\dim(M)=1$ is short: in this case, we have $M=S^1$ and our assumption on dimension means that $\dim(N) = 1$.  Simplicity of $\Diff^\infty_0(S^1)$ and the fact that $\Diff^\infty_0(S^1)$ is $C^r$-dense in $\Diff^r_0(S^1)$ means that finite-order rigid rotations act nontrivially on $N$, hence as finite order diffeomorphisms, and so $N = S^1$ and the orbit classification theorem means there is a single orbit for the action, which is conjugate to the standard action.   This can also be derived from the main theorem of \cite{Mann_ETDS}, which also gives a description of actions when $M$ is a noncompact 1-manifold, and does not assume continuity. 

The proof of Theorem  \ref{thm:structure} in the general case is somewhat involved, however it easily gives the negative solution to the extension problem, so we give this consequence first.  

\begin{proof}[Proof of Corollary \ref{cor:no_extension}]
Suppose $W$ is a compact smooth manifold with $\partial W = M$.  If $\res^r$ has a continuous section, then from Theorem \ref{thm:structure}, we know that there is a fiber bundle map $p: W\to M$ such that $p|_M=id$. This contradicts the fact that a boundary $\partial W$ is never a retract of $W$, which can be shown by the fact that the fundamental class $\mu_M\in H_{\dim(M)}(M;\mathbb{Z})$ has trivial image under the induced map of the embedding $M\to W$ into $H_{\dim(M)}(W;\mathbb{Z})$.

For $\res^\infty$, we do not have any section, since  \cite[Theorem 1.2]{Hurtado} implies that any extension $\Diff^\infty_0(M) \to \Diff^\infty_0(W)$ is automatically continuous.  
\end{proof} 

Now we establish the non-existence of fixed points.  For this, we do not need any restriction on the dimension of $N$, however we do use a hypothesis on the dimension of $M$.

\begin{prop}  \label{prop:no_fix2}
Let $M$ be a closed manifold with $\dim(M) \geq 3$.  If $\Diff_0^r(M)$ acts continuously on a manifold $N$ by $C^1$ diffeomorphisms with a global fixed point, then the action is trivial.
\end{prop} 

The proof uses the following easy observation about periodic points. 

\begin{obs} \label{obs:order_2}
Let $f$ be a local $C^1$ diffeomorphism of $\R^n$ fixing $0$ with $Df_0 = I$.   For any $k \in \mathbb{N}$, there exists a neighborhood $U_k$ of $0$ where $\Fix(f^k) \cap U_k \subset \Fix(f)$. 
\end{obs}

\begin{proof}[Proof of Observation \ref{obs:order_2}]
Let $k$ be given.   If $f^k$ has no fixed points in a neighborhood of $0$, then we are done.  
Otherwise, using continuity of $Df$ and the fact that it is identity at 0, let $U$ be a neighborhood of $0$ small enough so that, for any unit vector $v$ tangent to any point $y \in U$, and any $j < k$, we have that $||Df^j_y(v)||$ and $| Df^j_y(v) \cdot v |$ both have value in $[1/2, 3/2]$.  This choice is somewhat arbitrary, what is important here is that it is a small interval containing 1.    

Take a neighborhood $U_k \subset U$ small enough so that the convex hull of $U_k \cup f(U_k)$ lies in $U$.  Let $x \in \Fix(f^k) \cap U_k$, and suppose for contradiction that $f(x) \neq x$.  Let $L$ be a straight line segment from $x$ to $f(x)$ parametrized by unit speed, by construction this is contained in $U$.  Furthermore, for any $j = 1, 2, ... k-1$ the image $f^j(L)$ is a $C^1$ embedded curve from $f^j(x)$ to $f^{j+1}(x)$, with tangent vector at every point satisfying having norm in $[1/2, 3/2]$, and dot product with the vector in the tangent direction to $L$ also in $[1/2, 3/2]$.   It follows that the projection of this path to the line containing $L$ has positive derivative everywhere, so the union of the segments $f^j(L)$ cannot form a closed loop, contradicting the fact that $f^k(x) = x$.

%
%
\end{proof}

\begin{proof}[Proof of Proposition \ref{prop:no_fix2}]
Suppose $\rho: \Diff^r_0(M) \to \Diff^1(N)$ is a nontrivial, continuous action as given in the statement.  
By the orbit classification theorem, we know that every orbit of $\rho$ is a cover of a bundle over $\Conf_k(M)$ for some $k$ (depending on the orbit, but bounded in terms of the dimension of $N$).   We say such an orbit has {\em type k}. 

Suppose for contradiction that $\rho$ has at least one global fixed point 
 and let $F$ denote the fixed set.   For each $y \in F$, taking the derivative at $y$ gives a homomorphism $\Diff_0^r(M) \to \mathrm{GL}(n,\R)$ where $n = \dim(N)$.   Consider the restriction of this map to $\Diff_0^\infty(M)$; this is a simple group with no nontrivial finite dimensional linear representations, so the derivative homomorphism at any fixed point is trivial.  
Take some $y \in N$ that lies on the boundary of $F$ and let $y_i$ be a sequence of points in $N - F$ converging to $y$.  After passing to a subsequence, we can assume that each $y_i$ lies in a type of orbit of the same type, say type $l$.  
Thus, to each point $y_i$ we can associate a set $\pi(y_i) = X_i \subset M$ of cardinality $k$, such that $\rho(\Stab^r(X_i)_0) \subset \Stab(y_i)$.   

After passing to a further subsequence, we may assume that the sets $X_i$ Hausdorff converge to a closed set $X$; this will be a set of $k$ or fewer points on $M$.   By modifying the sequence $y_i$, we may also assume that $X_i \neq X$ holds for all $i$, as follows.  For each $i$ such that $X_i = X$, choose some $f_i \in \Diff^\infty_0(M)$ that is $C^r$ close to the identity, but with $f_i(X) \neq X$, and replace $y_i$ by $\rho(f_i)(y_i)$.  Then $\pi \rho(f_i)(y_i) = f_i(X) \neq X$, and since the action is continuous, we may choose $f_i$ close enough to identity so that  $\rho(f_i)(y_i)$ still converges to $y$.   Thus, we now assume no point in our sequence $y_i$ projects under the map $\pi$ to $X$. 

We now proceed to construct an element $g\in \Diff^r(M)$ such that $g^4\in \Stab^r(X_i)_0$ but $g(X_i)\neq X_i$. This will produce a contradiction with Observation \ref{obs:order_2} because $y_i$ is a periodic, but not fixed, point for $g$ for every $i$ and the limit $y$ is a global fixed point.
Take disjoint neighborhoods $U_x$ of each point $x \in X$.  Define a diffeomorphism $g \in \Diff^\infty_0(M)$ supported on the union of the sets $U_x$ as follows.   Fix an order two element $R \in \SO(m)$, where $m = \dim(M)$, and fix a path $\gamma$ of rotations from $R$ to the identity parametrized by $t \in [0, 1]$, that is constant on a neighborhood of $0$ and of $1$.  Take local charts identifying each $x\in X$ with the origin $0 \in \R^{d}$, and identifying a small neighborhood of $x$ contained in $U_x$ with the unit ball in $\R^m$.  Define $g$ to agree with the rotation $\gamma(t)$  on the sphere of radius $t$, and extend to the identity outside the ball, and use the chart to identify this with a diffeomorphism of $M$ supported on the union of the sets $U_x$.    For an appropriate choice of charts we can ensure that $g(X_i) \neq X_i$ for all large $i$.    

By construction, $g^2$ is the identity on a neighborhood of $X$, so for all large $i$, it acts trivially on the jet space over $X_i$.  
Since $g^2$ is described by a {\em loop of rotations} defined on concentric spheres (and based at identity) and $\pi_1(SO(m)) = \Z/2$, we can contract the loop of rotations corresponding to the element $g^4$ in $SO(m)$, and this gives an isotopy of $g^4$ to the identity supported on a compact set.  This shows that $g^4$ actually lies in $\Stab^r(X_i)_0$ for all $r$, so acts trivially on the fibers over $X_i$ coming from the fiber-bundle structure of each orbit.  In particular, $\rho(g)^4(y_i) = y_i$.  
However, since $g (X_i) \neq X_i$, we know that $\rho(g)(y_i) \neq y_i$ for all large $i$.    Since $y_i \to y$, and $D\rho(g)_y = I$, this gives the desired contradiction with  Observation \ref{obs:order_2}.  
\end{proof}
We suspect that a similar result holds for $\dim(M) = 2$, but the situation is more complicated because $\pi_1(SO(2)) = \Z$.  To avoid this, we instead simply treat the dimension 2 case separately for the purposes of Theorem \ref{thm:structure}.

\begin{proof}[Proof of Theorem \ref{thm:structure}]
We have already treated the easy case where $\dim(M)=1$.  We first give the proof in the higher dimensional setting, then show how to adapt the arguments to the surface case.  

\para{Case: $\dim(M) = m >2$}
By Proposition \ref{prop:no_fix2}, the action is fixed point free.  
As in the proof of Theorem \ref{flatbundle}, we first define a projection map $p: N\to M$ by setting $p(y)=x$ if $x$ is the unique point in $N$ such that $G_y \subset\Stab(x)$, where $G_y \subset \Diff_0^r(M)$ denotes the stabilizer of $y$ under the action of $\rho$.    This is well defined, since our restriction on dimension means that no orbits can be built from jet bundles over $\Conf_j(M)$ for any $j>1$.  
By Theorem \ref{thm:small_quotient_diff}, we also know that $\Stab^r(x)_0 \subset G_y$.  Also, the same argument as in the proof of Theorem \ref{flatbundle} shows that for any open ball $B \subset M$ we have
\[ 
p^{-1}(B)=N - \Fix(\rho(H(B)))
\]
where $H(B)$ denotes the diffeomorphisms supported on $B$ and isotopic to the identity through diffeomorphisms supported on $B$, so $p$ is a continuous map.    

Fix a basepoint $b \in M$ and let $F = p^{-1}(b)$.  
First, we define a local product structure as we did in the proof of the generalized flat bundle theorem for actions by homeomorphisms.  
Let $T$ be an embedding of a neighborhood of the identity in $\R^m$ into $\Diff^\infty_0(M)$ with $T(0) = \mathrm{id}$, such that $b$ has a free orbit under $T$, we again do this using a collection of $m$ smooth vector fields on $M$ which are linearly independent inside of a small coordinate box containing $b$, as we did in the proof of Proposition \ref{prop:cover}. The orbit map of $b$ under $T$ gives a smooth local chart for $M$ around $b$.  
Let $U$ be a small neighborhood of the identity in $T \subset \Diff^\infty_0(M)$.  Then the map  $\psi: U \times F \to p^{-1}(U \cdot b)$ defined by $\psi(t,y) = \rho(t)(y)$ is a continuous, injective map onto the open subset $p^{-1}(U\cdot b)$ of $N$. 

The inverse of this map sends a point $x = \rho(t_x)(y)$ to $(t_x, y) = (t_x, \rho(t_x)^{-1}(x))$.  The map $x \mapsto t_x$ is continuous, since $t_x$ is the unique point in $U$ such that $t_x b = p(x)$. Thus, this map is a homeomorphism onto its image, an open subset of $N$.  From this it follows that $F$ is a homology manifold of dimension $\dim(N)-\dim(M)$.

Note that $F$ may be disconnected, and may possibly even have infinitely many connected components, for instance in the case of lifting actions to a cover of a negatively curved manifold, or to the projectivized tangent bundle of such a cover.    

\begin{claim}\label{connectedh}
All connected components of $F$ are homeomorphic.  
\end{claim}
\begin{proof}
To do this, fix a connected component $C$ of $F$.  Let 
\[ D(C) = \bigcup_{f \in \Diff_0^r(M)} \rho(f)(C) \]
this is a $\rho$-invariant open subset of $N$ since $C$ is an open set in $F$.  Now if $\rho(f)(C) \cap F \neq \emptyset$ for some $f \in \Diff_0^\infty(M)$, then $f \in \Stab(b)$, which acts by homeomorphisms of $F$, permuting its connected components.  Thus, $\rho(f)(C)$ is either equal to $C$ or disjoint from it.  
It follows that the invariant sets $D(C')$, as $C'$ ranges over the connected components of $F$ in distinct orbits of the permutation action, form a partition of $N$ into countably many disjoint open sets.  Since $N$ is connected, this partition must be trivial, and we conclude that all connected components are homeomorphic.    
\end{proof}

Our next goal is to show that $F$ is a $C^s$ submanifold.   For this we use a different subgroup of $\Diff^r_0(M)$ supported in a neighborhood of $b$.  

Fix a local coordinate chart around $b \in M$, identifying $b$ with the origin in $\R^m$.  
Construct a subgroup $S \subset \Stab(b) \subset \Diff^\infty(M) \subset \Diff^r(M)$ of ``local rotations" as follows: 
choose a small neighborhood $\mathcal{U}$ of the identity in $\SO(m)$, and let $\phi_t$ be a smooth retraction of $\mathcal{U}$ to $\{id\}$.   Let $B_t$ denote the ball of radius $t$ about the origin in $\R^m$.  For $u \in \mathcal{U}$, take a diffeomorphism $g_u$ supported on $B_2$ agreeing with $u$ on $B_1$ and rotating the sphere of radius $1+t$ by $\phi_t(u)$.  Identify this with a diffeomorphism of $M$ by extending to the identity outside the image of $B_2$ in our coordinate chart.  Let $S$ be the group generated by $\{g_u:u\in \mathcal{U}\}$.  Since $S$ fixes $b$ and preserves the image of $B_1$ in our chart, $\rho(S)$ acts on $p^{-1}(B_1)$ and preserves the fiber $F$.  
\begin{claim}
The action of $S$ on $p^{-1}(B_1)$ factors through the universal cover $\widetilde{\SO(m)}=\text{Spin}(m)$ of $\SO(m)$.
\end{claim}
\begin{proof}
Let $g \in S$. For each $t \in [1,2]$, the restriction of $g$ to $B_t$ is a rotation, thus we can view $g$ as a path in $\SO(m)$ based at $id$ (the restriction to $B_2$) and ending at the restriction of $g$ to $B_1$.  
Any $g_1,g_2\in S$ satisfying $g_1|_{B_1}=g_2|_{B_1}$ are two paths to the same endpoint, so $(g_1g_2^{-1})$ defines a loop of rotations. As in the proof of Proposition \ref{prop:no_fix2}, since $\pi_1(\SO(m))=\mathbb{Z}/2$, we know that $(g_1g_2^{-1})^2\in \Stab^r(x)_0$ for any $x\in B_1$, so acts trivially on $p^{-1}(x)$.  Thus the action of $S$ on $p^{-1}(B_1)$  factors through $\widetilde{\SO(m)}$ by Theorem \ref{thm:small_quotient_diff} .
\end{proof}

Consider now the restriction of the action $\rho(S)$ to $F = p^{-1}(b)$.   
Suppose that $\mathcal{O} \subset F$ is an orbit of the action of $\rho(S)$ on $F$.  Then $\mathcal{O}$ is the continuous, injective image of a quotient of $\widetilde{\SO(m)}$ by some closed subgroup $H$.  Since $S$ is connected, $\mathcal{O}$ is connected, and our restriction on dimension implies that $\mathcal{O}$ has dimension at most $m-1$.  This dimension restriction means the closed subgroup is either $\widetilde{\SO(m)}$ itself (hence $\mathcal{O}$ is a point), or has identity component isomorphic to $\SO(m-1)$ or  $\widetilde{\SO(m-1)}$.  See \cite[sec.11]{montgomerysamelson} for the classification of small codimension closed subgroups of $O(m)$.  In the second case, $\widetilde{\SO(m)}/H$ is a manifold of dimension $m-1$ and if equipped with a metric induced from a bi-invariant metric on $\SO(m)$ has an isometry group of dimension $m(m-1)/2$ (the codimension of $h$) and thus by  \cite[Theorem II.3.1]{kobayashi} is a symmetric space, either $S^{m-1}$ or $\mathbb{R}\mathrm{P}^{m-1}$.

If some orbit is $S^{m-1}$ (respectively,  
$\mathbb{R}\mathrm{P}^{m-1}$), then our restriction on dimension implies that this orbit is necessarily a connected component of $F$.  Since all connected components are homeomorphic by Claim \ref{connectedh}, we conclude that $N$ is a $S^{m-1}$ (respectively,  $\mathbb{R}\mathrm{P}^{m-1}$) bundle over a cover of $M$; since the action is $C^s$ and the fiber an orbit of a compact group action, it is a $C^s$ submanifold \cite{MZ}.  

Otherwise, $F$ is pointwise fixed by $S$, and equal to $\Fix(S) \cap p^{-1}(B_1)$.   The local linearization theorem for actions of compact groups (see Theorem 1, \S 5.2 in \cite{MZ}) then says that this fixed set is a $C^s$ submanifold.   In all cases, we now know that $F$ is a $C^s$ submanifold and $\psi$ gives a local product structure.

\para{Case: $\dim(M)=2$}
Since Proposition \ref{prop:no_fix2} does not apply, here we work first with the action of $\rho$ on $N-\Fix(\rho)$ and then use its structure to show that $\Fix(\rho) = \emptyset.  $

Let $N'$ be a connected component of $N-\Fix(\rho)$.    
We apply the first part of the previous proof verbatim, taking a point $b \in N'$ and a smooth embedding $T$ of a neighborhood $U$ of the identity in $\R^2$ into $\Diff^\infty_0(M)$ such that $b$ has a free orbit under $T$.   As before, the map $U \times F \to p^{-1}(U\cdot b)$ defined by $(u, y) \mapsto \rho(u)(y)$ is continuous, injective, and a homeomorphism onto its image $p^{-1}(U\cdot b)$, an open subset of $N'$.  Thus, $F$ is a homology manifold, and our restriction on dimension implies that it is either a discrete set (in which case $N'$ is a cover of $M$ with the lifted action) or a one dimensional topological submanifold.  Claim \ref{connectedh} still applies and shows that all connected components of $F$ are homeomorphic. 

Also, since $T$ is a smooth embedding and $\rho$ is an action by $C^s$ diffeomorphisms, for any $y_0 \in F$, the orbit map $U \to N'$ given by $u \mapsto \rho(u)(y_0)$ is $C^s$, see \cite[Section 5.1]{MZ}.  

Following the proof from the previous case, we may define a group of local rotations $S$ such that the action of $S$ on $p^{-1}(B_1)$ factors through $\widetilde{\SO(2)}$.   Here it will be convenient to also require some compatibility between $S$ and $T$.   Working in the coordinates we used to define $T$, if $S$ is taken to agree with rigid rotations on the ball $B_1$ of radius $1$ about $0$ in $\R^2$, then for any $u$ sufficiently close to 0 in $U$, and any $s \in S$, the conjugate $sus^{-1}$ will also agree with a small rigid translation in some neighborhood of $0$.  In other words, the germ of $sus^{-1}$ at 0 agrees with the germ of some element $t \in T$.   Moreover, it follows from the construction of $S$ and $T$ that $t^{-1} \circ sus^{-1}$ is isotopic to the identity relative to a fixed neighborhood of $0$ in $\R^2$, and hence the diffeomorphism of $M$ which it defines is supported away from $b$ and isotopic to the identity relative to some neighborhood $B \subset B_1$ of $b$.  The orbit classification theorem then says that it acts trivially on $p^{-1}(B)$.    

This choice of $S$ means that, if $y_0 \in F$ is fixed by some $s \in S$, then the germ of $t$ and $sus^{-1}$ agree at $y_0$.  Thus, the tangent space to the embedded disc $\rho(U)(y_0)$ at $y_0$ is invariant under the action of $D\rho(s)_{y_0}$, for each $s \in S$, and the action of $S$ on the tangent space is by rotations.    This observation will be useful to us as we study the structure of $F$.   

Consider now the orbits of the action of $S$ on $F$.  These may be singletons, embedded circles on which $S$ acts by rotation, or copies of $\R$ on which $S$ acts by translation.
The difficulty in this case is that, unlike in the higher dimensional case, $S$ is not compact, and orbits may not be compact.  It is also not immediately evident that connected components of $F$ are either pointwise fixed or equal to orbits of $S$; one could in theory have a connected component of $F$ homeomorphic to $S^1$, but containing an orbit of $S$ homeomorphic to $\R$ together with a fixed point for $S$.  
To deal with this, we enlarge $S$ to contain local elements of $\SL(2,\R)$.  

Take a neighborhood $\mathcal{A}$ of the identity in the diagonal subgroup of $\SL(2, \R)$ and let $\mathcal{N}$ be a neighborhood of the identity in the nilpotent upper triangular subgroup, chosen small enough so that $v(B_1) \subset B_{3/2}$ for all $v \in \mathcal{A} \cup \mathcal{N}$ (recall $B_r$ denotes the ball of radius $r$ in $\R^2$).  Identifying $B_r$ with its image in $M$ via our chart, each $v \in \mathcal{A}$ gives a partially defined diffeomorphism of $M$ that can be extended to a diffeomorphism of $M$ that is identity outside of $B_2$, in a way that continuously embeds $\mathcal{A}$ into $\Diff_c^\infty(M)$, similarly for $\mathcal{N}$.  The group these diffeomorphisms generate consists of diffeomorphisms supported in $B_2$ and isotopic to the identity relative to $b$.  
Let $V \subset \Diff_c(M)$ be the group generated by all such elements.  The KAN decomposition implies that every germ of a linear map at $b$ is represented by some $s v$ where $s \in S$, and $v \in V$.  
The group generated by $S$ and $V$ fixes $b$ in $M$ so acts on $F = p^{-1}(b)$ in $N$.  This action factors through an action of the group $\widetilde{\SL(2,\R)}$, since $\widetilde{\SL(2,\R)}$ is topologically $\widetilde{SO(2)} \times N$.  

The classification of closed subgroups of  $\widetilde{\SL(2,\R)}$ implies that if $\mathcal{O}$ is a one-dimensional orbit for the action of $S$ on $F$, then it is also an orbit for the aciton of $\widetilde{\SL(2,\R)}$ and the action of $\widetilde{\SL(2,\R)}$on $\mathcal{O}$ is either the standard action of a finite cyclic cover of $\mathrm{PSL}(2,\R)$ on $S^1$ or of its universal covering group $\widetilde{SL(2,\R)}$ on $\R$.   In particular, the {\em standard action} means that, for any point $x$, there is a conjugate of a lift of the matrix $\left( \begin{smallmatrix} 2 & 0 \\ 0 & 1/2 \end{smallmatrix} \right)$ in $\SL(2,\R)$ which fixes $x$ and all of its translates under the deck group for the cover $\R \to S^1$, with derivative equal to 2 at these points, and also has a bi-infinite set of fixed points with derivative equal to 1/2 at each.

\begin{claim}
The closure $\overline{N'}$ of $N'$ in $N$ contains no global fixed point, hence $\Fix(\rho)=\emptyset$.  
\end{claim} 
\begin{proof} 
Suppose for contradiction that $\{x_n'\in N'\}$ is a sequence with $x_n' \to y \in \Fix(\rho)$.  Since $x_n'$ corresponds to a point $p(x_n')\in M$, we can assume that $p(x_n')\to b$ for some $b$ by passing to a subsequence. Then we can find elements $f_n\in \Diff^\infty_0(M)$, such that $f_n\to id$ and $f_n(p(x_n'))=b$. By continuity of the action, $\rho(f_n)(x_n')\to y$.  This shows that we can take a new sequence $x_n= \rho(f_n)(x_n')$ such that $x_n\in p^{-1}(b)$ and $x_n\to y\in \Fix(\rho)$.

Each $x_n$ lies in an orbit that is either an injective copy of $\R$, of $S^1$, or an isolated point.  After passing to a subsequence, we can assume all are the same type.   We treat these three cases separately, each will give us a contradiction.  

First suppose each $x_n$ is contained in an $\R$-orbit.  By compactness of $\R/\Z$, we may pass to a subsequence such that $x_n$ converges to some $\overline{x}$ mod $\Z$.   Then there exists $s_n\in S$ converging to identity such that $s_n(x_n) = \overline{x}$ mod $\Z$, and $\rho(s_n)(x_n)$ converges to $y$. 
Take $g \in V$ such that $\overline{x}$ is a fixed point of $g$ with derivative 2.   Then $\rho(g)$ also has derivative equal to 2 in the direction of the tangent space of the $\R$-orbit at the point $\rho(s_n)(x_n)$, contradicting the fact that these converge to a global fixed point, and at global fixed points all derivatives are trivial, as we remarked in the proof of Proposition \ref{prop:no_fix2}.   

The same argument shows that  no sequence of $S^1$ orbits for $S$ can accumulate to a global fixed point.   

Finally we deal with singleton orbits.  
Our observation that for any fixed point $y_0$ in $F$, the tangent space to the embedded disc $\rho(U)(y_0)$ at $y_0$ is invariant and rotated by $S$, means that $S$ acts with derivatives uniformly bounded away from identity at fixed points in $F$, so cannot accumulate to a global fixed point either.  Thus, no sequence of points in $F$ can converge to a point in $\Fix(\rho)$.   We conclude that $\Fix(\rho) = \emptyset$ and $N = N'$. 
\end{proof} 


To finish the proof of the theorem, we want to show that $\R$-orbits for $S$ are connected components of $F$, i.e. no $\R$-orbit for $S$ accumulates at a fixed point for $S$.  The strategy of the argument is similar to that above.  
Suppose for contradiction that $x_n \in F$ belong to a single $\R$-orbit $\mathcal{O}$ and $x_n \to y$ where $y$ is fixed by $S$.   Choose $s_n \in [0,1] \subset S \cong \R$ so that  the points $y_n := s_n(x_n) \in F$ all lie in the same orbit of $\tau$.  Since $y$ is fixed by $S$, we have $y_n \to y$.  
Recall that the tangent space to $y$ decomposes as a direct sum with a $S$-invariant plane, tangent to the image of $U(y)$, on which $S$ acts by rotations.   Let $v$ denote the axis of rotation in $T_y(N)$.  
Since $\mathcal{O}$ is an $S$-invariant set, the tangent  line to $\mathcal{O}$ at $y_n$ approaches $v$ as $n \to \infty$.  
Recall the description of the standard action of $V$ on $\R$ that we gave above.   Since $y_n$ all differ by a multiple of $\tau$,  we may take an element $g \in V$ (the lift of a hyperbolic element of $\SL(2,\R)$, as before) such that $\rho(g)$ fixes each point $y_n$, and has derivative equal to $1/2$ in the direction tangent to $\mathcal{O}$ there, and also fixes the points $s(y_n)$ with derivative equal to $2$ there, where $s \in S$ is some nontrivial element not equal to a multiple of $\tau$. Since $y$ is fixed by $s$, the sequence $s(y_n)$ also accumulates at $y$ and has tangent direction in $F$ tending to $v$.  This contradicts continuity of the derivative at $y$.  

Thus, each $\R$-orbit contained in $F$ is a connected component of $F$.  The same is true for the $S^1$-orbits.  It follows that the union the singleton orbits also forms a clopen set consisting of a union of connected components.     In the case of $S^1$ or $\R$ orbits, these are $C^s$ embedded submanifolds, and we have a local product structure using the group $T$ and chart $\psi$ as before.   In the neighborhood of a singleton orbit, the action of $S$ on $p^{-1}(B_1)$ factors through an action of a compact group, since the deck transformation $\tau$ of $S \to SO(2)$ has support disjoint from $B_1$, so acts trivially on the fibers over all points of $B_1$.   We may now apply local linearization to conclude these components of $F$ are $C^s$ embedded submanifolds as well.  \qedhere

\end{proof} 
%
%
%
%
%
%

\section{Application: actions of $\Homeo_0(S^1)$ on compact surfaces}  \label{sec:Militon}
In this section, we will classify actions of $\Homeo_0(S^1)$ on compact surfaces. We will give the proof for the disc $\mathbb{D}^2$ which proves \cite[Conjecture 2.2]{Militon}.   A complete classification of actions on other surfaces can be obtained by essentially the same  argument, giving in particular a new proof of the main theorem of \cite{Militon}. 

\subsection{The statement}
We first give a general procedure to construct actions of $\Homeo_0(S^1)$ on $\mathbb{D}^2$.  (Similar to \cite[\S 2]{Militon}.)
Let $L=[0,1]$ be the orbit space of the standard $SO(2)$ action on $\mathbb{D}^2$, where $r\in [0,1]$ represents the circle of radius $r$. Let $K\subset L$ be a closed subset including $0$. We use the convention $S^1=\mathbb{R}/\mathbb{Z}$ in the following. Let  $a_0:\Homeo_0(S^1)\to \Homeo(\mathbb{R}/\mathbb{Z} \times [0,1])$ be defined by
 \[
a_0(f) (\theta,r)=(f(\theta), \tilde{f}(r+\tilde{\theta})-\tilde{f}(\tilde{\theta}))
\]
where $\tilde{f}\in \Homeo_0(\mathbb{R})$ is any lift of $f\in \Homeo_0(S^1)$ to $\Homeo_0(\R)$ and $\tilde{\theta}\in \mathbb{R}$ represents a lift of $\theta \in \mathbb{R}/\mathbb{Z}$.  Note that this is well defined and independent of the choice of lifts. 

Let $T^k(\theta, r)=(\theta+kr,r)$, this is the $k$th power of a standard Dehn twist in the closed annulus $\mathbb{R}/\mathbb{Z} \times [0,1]$.  Let $a_k(f)=T^ka_0(f)(T^k)^{-1}$. For example, 
\[
a_1(f)(\theta,r)=T^1a_0(f)(\theta-r,r)=
T^1(f(\theta-r),\tilde{f}(\tilde{\theta})-\tilde{f}(\tilde{\theta}-r))=
(f(\theta),\tilde{f}(\tilde{\theta})-\tilde{f}(\tilde{\theta}-r)).
\]
The fact that both $a_0(f)$ and $a_1(f)$ have the same first coordinate is a pure coincidence.  This is not true for $a_k$ when $k\neq 0,1$, and so $a_0$ and $a_1$ will play a special role. 

Let $\lambda:L-K\to \{0,1\}$ be a function which is constant on each component. For $a<b \in [0,1]$, let $n_{a,b}$ be the affine normalization $n_{a,b}: S^1\times [0,1]\to S^1 \times [a,b]$ given by $n_{a,b}(\theta,r)=(\theta, a+r(b-a))$. We denote by $\rho_{K,\lambda}$ the action such that, for each component $(a,b)$ of $L-K$, the restriction of $\rho_{K, \lambda}$ to its closure is given by
\[
\rho_{K,\lambda}|_{S^1\times [a,b]}(f)=n_{a,b}\circ a_{\lambda((a,b))}(f) \circ n_{a,b}^{-1}.
\]
It is easy to check that $\rho_{K,\lambda}$ is indeed a continuous group action, since the first coordinate of $a_0$ and $a_1$ is just the standard action on $S^1$, and the second coordinate is also continuous. 

We prove the following.

\begin{thm}[Classification of $\Homeo_0(S^1)$ actions on the disc]\label{Mil}
Any nontrivial homomorphism $\rho: \Homeo_0(S^1)\to \Homeo_0(\mathbb{D}^2)$ is conjugate to $\rho_{K,\lambda}$ for some $K,\lambda$ as above. Two homomorphisms $\rho_{K,\lambda}$ and $\rho_{K',\lambda'}$ are conjugate to each other if and only if there is a homeomorphism $h:[0,1]\to [0,1]$ with $h(K)=K'$, and such that $\lambda$ and $h^{-1}\circ \lambda' h$ agree on all but finitely many components of $[0,1]-K$ on each closed interval in the interior of $[0,1]$.
\end{thm}
The reader may wonder why $\lambda$ only takes values in $\{0,1\}$. This is because that $\{a_k\}$ are all conjugate to one another, which means we can perform a conjugation supported on finitely many annuli to convert them all to $a_0$. However, as we will show later, we cannot simultaneously conjugate infinitely many such maps over a compact set $S^1 \times [r,1] \subset S^1 \times (0,1]$, and only finitely many of them can take values outside of $\{0,1\}$.

The proof will also show that actions on the half-open annulus (up to conjugacy) are the same as those on the disc under the identification of $[0,1) \times S^1$ with $\DD^2 - 0$.   Actions on the open and closed annulus, the sphere, and the torus have an analogous classification which can be obtained by the same proof.  

We will use the following classical result on $SO(2)$ actions. 

\begin{lem}[\cite{MZ} Ch 6.5] \label{lem:SO2}
Any faithful, continuous action of $SO(2)$ on $\mathbb{D}^2$ is conjugate to the standard action by rotations.
\end{lem}
\noindent This is also true for the sphere and the (open, closed, or half-open) annulus, while all actions of $SO(2)$ on the torus $S^1 \times S^1$ are conjugate to rotation of one $S^1$ factor.  

%
%
%

Now suppose that $\rho: \Homeo_0(S^1) \to \Homeo_0(\DD^2)$ is a representation.  Using the automatic continuity result of Rosendal and Solecki \cite{RS}, we  know that $\rho$ is continuous, and by simplicity of $\Homeo_0(S^1)$ we may assume $\rho$ is faithful.  By Lemma \ref{lem:SO2}, we can also assume that the restriction of $\rho$ to $SO(2)$ agrees with the standard action by rotations.  
We will apply several successive conjugations to put $\rho$ in the form stated in Theorem \ref{Mil}. 

\subsection{First conjugation: coning on a closed, invariant set}
Fix $s \in S^1$ and let $G_s$ denote the stabilizer of $s \in \Homeo_0(S^1)$.  
Later, we will use an identification of $S^1$ with $\R/\Z$ and take $s$ to be $0$ in this identification.  
Thinking of $\DD^2$ as the unit disc in $\R^2$, the first conjugacy will put $\Fix(\rho(G_s))$ on the $x$-axis, so that the restriction of $\rho$ to the set $SO(2)(\Fix(\rho(G_p)))$ agrees with coning.  

\begin{lem}
The fixed point $O$ of $\rho(SO(2))$ is a global fixed point for $\rho(\Homeo_0(S^1))$.
\end{lem}
\begin{proof}
This lemma has a direct proof which is given in \cite[Proposition 6.1]{Militon}.  An alternative quick argument can be obtained by quoting the general classification \cite[Theorem 1.1]{GM}, since  the stabilizer $\Stab(O)$ of $O$ under $\rho$ contains $SO(2)$ (in particular, it contains a nonconstant path), and also contains any element that commutes with a nontrivial element of $SO(2)$. 
\end{proof}

Let $L=[0,1]$ be the orbit space of the standard $SO(2)$ action on $\mathbb{D}^2$, where $r\in [0,1]$ represents the circle of radius $r$, and let $p$ denote the projection $p: \mathbb{D}^2\to L=[0,1]$.
With the exception of $O$, every other $\rho(\Homeo_0(S^1))$-orbit is 1 or 2-dimensional, since nontrivial orbits under the subgroup $\rho(SO(2))$ are 1-dimensional.  Let $Q\subset \mathbb{D}^2$ be the union of all the 1-dimensional orbits and $O$. Then $K:=p(Q)\subset [0,1]$ is a closed subset. (Each 2-dimensional orbit is open by invariance of domain.) The action of $\Homeo_0(S^1)$ is standard on each 1-dimensional orbit, so the point stabilizer $G_s$ fixes a single point.  Thus, we may define
 $h_1': K - \{0\} \to \mathbb{D}^2$ by 
\[h_1'(x)=\Fix(\rho(G_s))\cap p^{-1}(x).\] Away from $O$, this function is uniquely determined by its projection to the $S^1$ coordinate, which we denote by $g_1':K-\{0\}\to S^1$.
\begin{lem}
The functions $h_1'$ and $g_1'$ are continuous on $K-\{0\}$.
\end{lem} 
\begin{proof}
Continuity of $h_1'$ implies continuity of $g_1'$. For $h_1'$, we need to show that for a closed set $C\subset \mathbb{D}^2$, the set $h_1'^{-1}(C)$ is closed. We have the following computation:
\[
h_1'^{-1}(C)=p(\Fix(\rho(G_s))\cap C).
\]
Since $\Fix(\rho(G_s))$ is a closed set and $p$ is a proper map, we know that $h_1'^{-1}(C)$ is a closed set as well. 
\end{proof}
By the Tietze extension theorem, we can extend the function $g_1'$ to a continuous function $g_1: (0,1]\to S^1$. 
Identify $S^1 = \mathbb{R}/\mathbb{Z}$, so $s$ is identified with $0$, and $G_s = G_0 \subset \Homeo_0(S^1)$.   
Since $\rho(G_0)$ fixes $g_1(r)$ for $r\in K$, the action of $\rho(f)$ on $p^{-1}(r)$ agrees with $\rho(f)(\theta,r)=(f(\theta-g_1(r))+g_1(r),r)$.  Define a homeomorphism $h_1$ of $\mathbb{D}^2$ by 
\[
h_1(\theta,r)=(\theta+g_1(r),r)\]
and $h_1(O)=O$.  
Then $h_1^{-1} \circ \rho \circ h_1|_{p^{-1}(K)}$ is ``coning", i.e. $h_1^{-1} \circ \rho(f) \circ h_1(\theta,r)=(f(\theta),r)$ whenever $r\in K$. From now on, we replace $\rho$ with its conjugate $h_1^{-1} \circ \rho \circ h_1$.

\subsection{Building block: Indecomposable actions}
Call an action of $\Homeo_0(S^1)$ on a surface \emph{indecomposible} if there are no zero or 1 dimensional orbits.  The following is an easy consequence of Theorem \ref{thm:orbit_classification}. 
\begin{cor} \label{cor:irred}
Up to conjugacy, there are only two indecomposible actions of $\Homeo_0(S^1)$ on connected surfaces: the standard action on $\Conf_2(S^1)$ and the standard action on $\PConf_2(S^1)$, which is the space of ordered pair of points on $S^1$.
\end{cor}
\begin{proof}
Let $S$ be a connected surface with an indecomposible $\Homeo_0(S^1)$ action. Since $S$ is 2-dimensional, by invariance of domain every orbit is an open subsurface of $S$, and by Proposition \ref{lift}, each is homeomorphic to either $\PConf_2(S^1)$ or $\Conf_2(S^1)$. Since $S$ is connected, it cannot be covered by disjoint open subsurfaces, so is either $\PConf_2(S^1)$ or $\Conf_2(S^1)$.
\end{proof}

Since $\Conf_2(S^1)$ is homeomorphic to the M\"obius band, this orbit type does not occur when the surface is orientable.  So we may focus on $\PConf_2(S^1)$, which is homeomorphic to an open annulus.   These are precisely the 2-dimensional orbits that occur in Theorem \ref{Mil}.  

\begin{claim} \label{claim:conj} 
The conjugation between the action $a_0$ and $\Homeo_0(S^1)$ on $\PConf_2(S^1)$ is given by the homeomorphism $H: \PConf_2(S^1)\to S^1\times (0,1)$ such that $H(x,y)=(x,y-x)$ where $y-x$ is the distance between $x,y$ along the anti-clockwise direction.
\end{claim}
\begin{proof}
To verify the conjugation, for $f\in \Homeo_0(S^1)$, we need to show that $ H\circ f = a_0(f)\circ H$. For $(x,y)\in \PConf_2(S^1)$, we have $H\circ f(x,y)=(f(x),f(y)-f(x))\in S^1\times (0,1)$.   We also have
\[
a_0(f)\circ H(x,y)= a_0(f)(x, y-x) = (f(x), \tilde{f}(\tilde{x}+(y-x))-\tilde{f}(\tilde{x})),
\]
Now $\tilde{x}+y-x$ is a lift of $y$.  Denoting this by $\tilde{y}$ we have 
\[
a_0(f)\circ H(x,y)=  (f(x), \tilde{f}(\tilde{y} -\tilde{f}(\tilde{x})) = (f(x),  f(y)-f(x) \in S^1 \times [0,1].
\] 
Thus, we know that these two actions are conjugate. 
\end{proof}

\subsection{Second conjugation: 2-dimensional orbits}  \label{subsec:2nd} 
Let $S^1\times (a,b)$ be a $\rho$-invariant open annulus on which $\rho$ is irreducible.  By Corollary \ref{cor:irred}, there exists $h\in \Homeo(S^1\times (a,b))$ such that 
\[
h\circ \rho(f)\circ h^{-1}=n_{a,b} \circ a_0(f)\circ n_{a,b}^{-1}.
\]
We also have $\rho(r_\theta)=r_\theta=a_0(r_\theta)$, and therefore
\[
h\circ r_\theta=r_\theta\circ h.
\]
Writing $h$ in coordinates as $h(\alpha,r):=(g(\alpha,r),k(\alpha,r))$, the coordinate functions satisfy
\[
g(\alpha+\theta,r)=g(\alpha,r)+\theta,
\]
\[
k(\alpha+\theta,r)=k(\alpha,r)
\]
for any $\theta$. This shows that $k(\alpha,r)=k(r)$ and $g(\alpha,r)=\alpha+g(r)$ where $g(r)$ and $k(r)$ are continuous functions of $r \in [a,b]$ such that $g(a)=g(b)=0$ and $k(a)=a$ and $k(b)=b$. Since for each 2-dimensional orbit of the form $S^1\times (a,b)$, we have that $k(a)=a,k(b)=b$ and $k:[a,b]\to [a,b]$ is increasing, these $k$ glue together to give a continuous function, which extends to a continuous function that is the identity outside of the union of the 2-dimensional orbits of $\rho$.  Abusing notation, denote this function also by $k$, and let $h_2(\theta,r)=(\theta,k(r))$.  This defines a homeomorphism of $\mathbb{D}^2$. 
Going forward, we replace $\rho$ with its conjugate $h_2\circ \rho \circ h_2^{-1}$.
This simplifies the form of the associated function $h$ conjugating $\rho$ to the $a_0$ action given by Corollary \ref{cor:irred}, and we now have 
\begin{equation} \label{normalform}
h(\alpha,r)=(g(r)+\alpha,r).
\end{equation} 

Say that a 2-dimensional orbit has {\em degree d} if this map $g:[a,b]\to S^1$ is isotopic relative to the boundary to a degree $d$ map (Recall that $g(a)=g(b)=0$, so the notion of degree makes sense here).  Our next goal is to prove the following.  
\begin{lem}\label{lem:2nd}
Let $J \subset (0, 1]$ be a closed interval bounded away from $0$.  Among the 2-dimensional orbits that intersect $J \times S^1$, only finitely many have degree not in $\{0, 1\}$. 
\end{lem}

We will use the notation $G_\theta$ for the stabilizer of $\theta \in S^1 = \R/\Z$.  There are two 1-dimensional orbits of $G_0$ in $\PConf_2(S^1)$.  
For the $a_0$ action, as in Claim \ref{claim:conj}, these orbits are $\{(-r,r) \mid r\in (0,1)\}$ and $\{(0,r)\mid r\in (0,1)\}$.   Thus, the 1-dimensional orbits of $\rho(G_0)$ on an irreducible annulus $S^1\times (a,b)$ are the sets $\{h(0,r) \mid  r\in (0,1)\}$ and $\{ h(n_{a,b}(-r,r)) \mid r \in (0,1)\}$.

\begin{proof}[Proof of Lemma \ref{lem:2nd}] 
Suppose for contradiction that there are infinitely many 2-dimensional orbits $U_k = S^1 \times (a_k, b_k)$, indexed by $k \in \mathbb{N}$, of degree different from $0$ or $1$ inside a compact sub-annulus of $\DD^2 - O$.   Without loss of generality, we assume they all have degree $>1$ (the case where the degree is negative is similar) and assume that $a_k$ converges monotonically to some $r \in (0,1)$.  

For a fixed degree $d$ orbit $U = S^1 \times (a, b)$, the 1-dimensional orbits of $\rho(G_0)$ in $U$ are the sets 
\[
\{h\circ n_{a,b}(-r,r)|r\in (0,1)\}=\{(g(t),t)|t\in (a,b)\} \]
and\[
\{h\circ n_{a,b}(-r,r)|r\in (0,1)\}=\{(g(t)-\tfrac{t-a}{b-a},t)|t\in (a,b)\}
\]
where $g$ is a degree $d$ map, and hence the map $t \mapsto g(t)-\frac{t-a}{b-a}$ has degree $d-1 \geq 1$.  

For convenience, we now switch to working on the universal cover.   
Since $\rho(G_0)$ acts on $\DD^2 - O$ with fixed points, we may lift it to an action $\tilde{\rho}$ with fixed points on the universal cover of $\DD^2 - O$; then the 1-dimensional orbits in $U_k$ are continuous curves from $(m, a_k)$ to $(m+d, b_k)$ and from $(m, a_k)$ to $(m+d-1, b_k)$,  for $m \in \Z$.  See Figure \ref{fig:G0fig}. 

  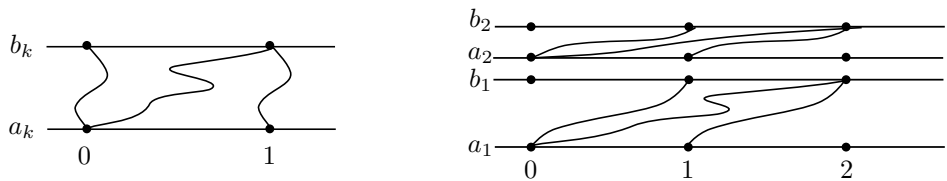
\begin{figure}[h]
   \labellist 
  \small\hair 2pt
     \pinlabel $a_k$ at 2 25
    \pinlabel $b_k$ at 2 67
   \pinlabel $0$ at 33 12
    \pinlabel $1$ at 125 12
       \pinlabel $0$ at 255 5
    \pinlabel $1$ at 333 5
        \pinlabel $2$ at 412 5
             \pinlabel $a_1$ at 230 16
    \pinlabel $b_1$ at 230 50
                 \pinlabel $a_2$ at 230 62
    \pinlabel $b_2$ at 230 80
   \endlabellist
     \centerline{ \mbox{
 \includegraphics[width = 5in]{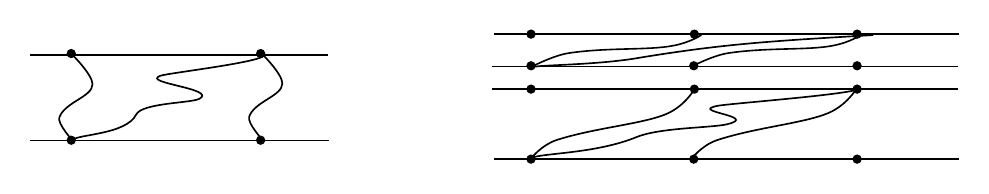}}}
 \caption{Orbits of $G_0$ in a degree 1 orbit for $\rho$ (left) and in degree 2 orbits (right). Dots represent fixed points.}
  \label{fig:G0fig}
  \end{figure}

Each $1$-dimensional orbit of $G_0$ is canonically homeomorphic with $(0,1)$  via a map $\phi: \theta \mapsto \Fix(G_0 \cap G_\theta)$.  (The map $\phi$ depends on the orbit, but for the sake of readability we suppress that notation for the time being.)
Choose $\theta \in (0,1)$ and, let $x_k \in \{\theta\} \times (a_k, b_k)$ be a point contained in a 1-dimensional orbit of $G_0$. Let $\theta_k = \phi^{-1}(x_k)$.   Pass to a subsequence so that $\theta_k$ converges to some $\theta_\infty \in S^1$.   If $\theta_\infty \neq \theta$, then we may take $h \in G_0$ such that $h(\theta) \neq \theta$, but $h(\theta_k) = \theta_k$ for all $k$ large.   But since $\Fix(\rho(h))$ is a closed set, it contains the point $(\theta, r)$, contradicting the fact that $h \notin G_\theta$.   Thus, we conclude that $\theta_\infty = \theta$, i.e. every convergent ``vertical" sequence of points $x_k = (\theta, r_k)$ contained in 1-dimensional orbits of $G_0$ corresponds, via the identifications defined by $\phi$, to a sequence converging to the first coordinate $\theta$. 
But this contradicts the existence of a sequence of orbits of degree $>1$.   
\end{proof}

\subsection{Final conjugation: normal form for 2-dimensional orbits} 
We now conjugate the map on each 2-dimensional orbit so that it agrees on each orbit with $a_0$ or $a_1$.
First, by Lemma \ref{lem:2nd}, there are only finitely may orbits of degree neither $0$ nor $1$.  Let $h_3$ be a homeomorphism of $\mathbb{D}^2$ that is supported on the union of these orbits and agrees with a power of a Dehn twist on each, conjugating the action on this orbit to the (normalized) action of $a_0$.  

Let $h_4:\mathbb{D}^2\to \mathbb{D}^2$ be a function that is the identity map outside the union of the degree 0 orbits and is the conjugation map that conjugates each degree 0 orbit action to $a_0$-action (this is the function $h$ from equation \eqref{normalform} in Section \ref{subsec:2nd}), and let $h_5:\mathbb{D}^2\to \mathbb{D}^2$ be a similar function for degree 1 orbits like $h_4$. We need to show that $h_4$ and $h_5$ are homeomorphisms; we give the details for $h_4$, the case of $h_5$ is completely analogous.     

By construction, $h_4$ is continuous on each individual 2-dimensional orbit, and extends to a continuous function on the closure of each individual orbit.   What we need to show is continuity at accumulation points of such orbits.  Suppose that $S^1 \times (a_n,b_n)$ is a sequence of $2$-dimensional orbits of degree 0 with $a_n \to r$ for some $r \neq 0$.   Let $g_n$ denote the function from equation \eqref{normalform} on $S^1 \times (a_n,b_n)$, and $\phi_n$ the identification with $(0,1)$ defined in the proof of Lemma \ref{lem:2nd} (where it was called $\phi$).  We need to show that $g_n$ converges to the constant function 0.   Recall that $c_n:=\{(g_n(t)-\frac{t-a_n}{b_n-a_n},t)|t\in (a_n,b_n)\}$ is a 1-dimensional orbit.  If $g_n$ did not converge to $0$, we could pass to a subsequence and find a sequence of points $t_n = \phi_n(\theta)$ such that $|g_n(t_n)| > \epsilon >0$ where $\phi_n$ denotes the corresponding identification $\phi_n: \theta \to \Fix(G_0\cap G_\theta)\cap c_n$.  But this contradicts the fact that $\phi_n(\theta)$ converges to $(\theta,r)$ as shown in the proof of Lemma \ref{lem:2nd}.  Thus, $h_3 \circ h_4 \circ h_5$ is a homeomorphism conjugating $\rho$ into standard form.  

\subsection{Characterization of conjugacy classes}
It remains only to show that two homomorphisms $\rho_{K,\lambda}$ and $\rho_{K',\lambda'}$ are conjugate to each other if and only if there is a homeomorphism $h:[0,1]\to [0,1]$ such that $h(K)=K'$ and $\lambda$ and $h\circ \lambda'$ agree on all but finitely many components of $[0,1]-K$ on each closed interval in the interior of $[0,1]$.  
This is proved in \cite[Proposition 2.3]{Militon}.   In brief, since $K$ and $K'$ are the union of 1-dimensional orbits, they are necessarily conjugate if the actions are.  On any fixed two-dimensional orbit, there is a unique conjugacy between the two actions by Corollary \ref{cor:irred}, and if $\lambda$ and $h\circ \lambda'$ differ on only finitely many components of $[0,c]-K$ for some $c<1$, then these conjugacies can glue together to form a continuous homeomorphism.   Conversely, if there is a conjugacy between the actions, one can identify the (necessarily finitely many) components on which they differ by looking at the image of a radial line under the conjugacy.
Full details can be found in \cite{Militon}. 
%
%


\bibliographystyle{plain}
\bibliography{biblio}

\end{document}